\documentclass[11pt,reqno]{amsart}
\usepackage{amsmath, amsfonts, amsthm, amssymb, color}
\usepackage{setspace}
\usepackage{amsthm,amsmath,amssymb}
\usepackage{mathrsfs}
\usepackage{geometry}
\usepackage{amsfonts}
\usepackage[utf8]{inputenc}
\usepackage{amssymb}
\usepackage{amsmath}
\usepackage{subfigure}
\usepackage{graphicx}
\usepackage{amsmath,amscd}
\everymath{\displaystyle}
\usepackage{indentfirst} 
\usepackage{enumerate}
\usepackage[colorlinks]{hyperref}
\usepackage{bm}
\usepackage{algorithm}
\usepackage{algorithmic}
\usepackage{mathrsfs}
\usepackage[toc,page,title,titletoc,header]{appendix}
\newtheorem{prop}{Proposition}[section]
\newtheorem{rmk}{Remark}[section]
\newtheorem{defi}{Definition}[section]
\newtheorem{lemma}{Lemma}[section]
\newtheorem{thm}{Theorem}[section]

\newtheorem{cor}{Corollary}[section]

\usepackage{graphicx} % Required for inserting images
\numberwithin{equation}{section}

\newcommand{\ud}{\,\mathrm{d}}

\newcommand{\R}{{\mathbb{R}}}
\newcommand{\M}{{\mathcal{M}}}
\newcommand{\E}{{\mathbb{E}}}

\author{Xuanrui Feng*}
\address{School of Mathematical Sciences, Peking University, Beijing 100871, China.}
\email{pkufengxuanrui@stu.pku.edu.cn}

\author{Zhenfu Wang}
\address{Beijing International Center for Mathematical Research, Peking University, Beijing 100871, China}
\email{zwang@bicmr.pku.edu.cn}

\title[Propagation of Chaos]
{Quantitative Propagation of Chaos for 2D Viscous Vortex Model with General Circulations on the Whole Space}
\subjclass[2020]{35Q35, 76N10}
\keywords{Viscous Vortex Model, Mean Field Limit, Propagation of Chaos, Logarithmic Growth Estimates}
\date{\today}

\keywords{Viscous Vortex Model, Mean Field Limit, Propagation of Chaos, Logarithmic Growth Estimates}

\begin{document}

\begin{abstract}
    We derive quantitative propagation of chaos in the sense of relative entropy for the 2D viscous vortex model with general circulations, approximating the vorticity formulation of the 2D Navier-Stokes equation on the whole Euclidean space. Our results work on the general setting that the vortices are positioned on the whole space $\R^2$ and that the circulations are allowed to be in different magnitudes and orientations, which can be adapted to general unconfined realistic fluids with vorticity that may change sign. We provide explicit convergence rates which are optimal in $N$ and optimal in $t$ among existing literature. The key technical tools, which are our major novelty, are the sharp logarithmic growth estimates and a new ODE hierarchy and iterated integral estimates.
\end{abstract}

\maketitle

\tableofcontents

\section{Introduction}

Consider the $N$-particle stochastic point vortex system positioning on the whole Euclidean space $\R^2$ with general circulations for vortices, whose evolution is governed by the following stochastic differential equations (SDEs):
\begin{equation}\label{particle}
    \ud X_i(t)=\frac{1}{N} \sum_{j \neq i} \M_j \cdot K(X_i-X_j)\ud t +\sqrt{2\sigma} \ud B_i(t), \quad 1 \leq i \leq N.
\end{equation}
The particles are labeled with the index $i$ ranging from $1$ to $N$,  with $X_i(t) \in \R^2$ and $\M_i/N \in \R$ representing, respectively, the position and circulation of the $i$-th vortex. 

The circulation $\M_i/N$ is an important quantity in the point vortex model. It measures the strength of the $i$-th vortex. The different cases of circulations with $\M_i>0$ and $\M_i<0$ represent the two different orientations of the point vortices. It is a very common physical phenomenon in real-world flows to have vortices of different signs. A typical example is a vortex dipole, a pair of closely spaced vortices with circulations of equal magnitude but opposite sign. We mention that the scale $1/N$ that appears in the vortex circulation is in the sense of \textit{mean-field scaling}, keeping the whole circulation of the system bounded.

The displacement of the particles is determined by the combination effect of two factors. The first one is given by the pairwise interaction force with some weights given by the circulations of other particles. The interacting kernel $K$ has many different options arising from various natural physical models. In our study, this kernel is given by the famous 2D Biot-Savart law in fluid mechanics,
\begin{equation*}
    K(x)=\nabla^\perp \log |x|=\frac{1}{2\pi} \frac{(-x_2,x_1)}{|x|^2},
\end{equation*}
where $x=(x_1,x_2) \in \R^2$. The second one is given by the extra additive noise, driven by $N$ independent copies of the standard 2D Brownian motion, which are denoted by $\lbrace B_i(t) \rbrace_{1 \leq i \leq N}$. The constant $\sigma>0$ represents the viscosity of the system dynamics and measures the intensity of the diffusion.

The point vortex model \eqref{particle} has significant physical importance and meanings, especially in fluid mechanics and statistical mechanics. In fluid mechanics, the motion of a viscous incompressible fluid is classically described by the Navier-Stokes equation. The notion of vorticity and the corresponding vorticity formulation of the 2D Navier-Stoke equation, which describe the local rotation of the velocity field, are powerful and intuitive to understand complex fluid flows in meteorology or oceanography, for instance. To better understand the vorticity function, a long-standing and fruitful approach in mathematical physics and applied mathematics is to approximate the vorticity equation by a discrete interacting point vortex system. The discretized version of the vorticity is useful and effective in both physically modeling and numerically simulating the 2D Navier-Stokes equation. 

Such point vortex model was initiated in the study of Helmholtz \cite{helmholtz1858uber} and Kirchhoff \cite{kirchhoff1876vorlesungen}. Mathematically, it was Osada \cite{osada1986propagation} who first showed that, although the kernel has some singularity of $O(1/|x|)$ at the origin, which makes the interaction between two particles large enough when they become sufficiently close to each other, the particles actually never collide in the evolution process with probability one. Therefore, the singularity of the system will never be reached, and the well-posedness of \eqref{particle} is justified. This allows us to use the convention that $K(0)=0$, and thus add the terms with $i=j$ to the summation in \eqref{particle}. 

Throughout this article, we make the common assumption that $\M_i$ is independent of time $t$, which means that the circulation of each point vortex, once determined at the initial time $t=0$, remains unchanged as the system evolves in time. We also assume that $\lbrace \M_i \rbrace_{1 \leq i \leq N}$ are $N$ independent copies of some random variable $\M$ compactly supported in $[-A,A]$ for some $A>0$. Typical examples include the uniform distribution $U(-A,A)$. Our assumption made on $\M_i$ corresponds to the practical statistical procedure that we put some i.i.d. circulations on each point vortex to expect the convergence of the vorticity as $N \rightarrow \infty$ in some Law of Large Numbers sense.

\begin{rmk}
    We remark here that without loss of generality we can assume that $A \leq 1$. Indeed, for any general case, we perform the simple scaling $Y_i(t)=X_i(t)/\sqrt{A}$ to the particle system \eqref{particle} to obtain that
    \begin{equation}\label{scaleparticle}
        \ud Y_i(t)=\frac{1}{N} \sum_{j \neq i} \frac{\M_j}{A} \cdot K(Y_i-Y_j)\ud t +\sqrt{\frac{2\sigma}{A}} \ud B_i(t), \quad 1 \leq i \leq N.
    \end{equation}
    This new point vortex system \eqref{scaleparticle} is endowed with scaled circulations $\Tilde{\M}_j=\M_j/A$, which are supported in $[-1,1]$. Therefore, we can always change our target system to satisfy the additional assumption $A \leq 1$, at the price that the viscosity constant $\sigma$ is also rescaled by the constant $A$, which is consistent with our assumptions on $\sigma$ in the main results.
\end{rmk}

We are interested in the mean-field limit dynamics of the point vortex system \eqref{particle} when $N \rightarrow \infty$ in some appropriate sense of statistical mechanics. It is well-known in the mean-field theory that the limit self-consistent McKean-Vlasov system writes
\begin{equation}\label{limit}
    \ud X_t=(K \ast \omega_t)(X_t)+\sqrt{2\sigma}\ud B_t,
\end{equation}
where $X_t \in \R^2$ and $B_t$ is a standard 2D Brownian motion, and the vorticity function
\begin{equation*}
    \omega_t(x)=\omega(t,x)=\E_\M g_t(\M, X_t)=\int_\R mg_t(m,x) \ud m
\end{equation*}
is the expectation of $g_t(m,x)$ with respect to $m$, where
\begin{equation*}
    g_t(m,x)=\text{Law}(\M,X_t)
\end{equation*}
represents the joint density of $\M$ and $X_t$. By the It\^o formula, it is straightforward to check that the vorticity function $\omega_t$ satisfies the famous vorticity formulation of the 2D Navier-Stokes equation in fluid mechanics:
\begin{equation}\label{vorticity}
    \partial_t \omega_t +(K \ast \omega_t) \cdot \nabla \omega_t=\sigma \Delta \omega_t.
\end{equation}

Therefore, once we rigorously justify the mean-field convergence from the particle system \eqref{particle} to the limit McKean-Vlasov system \eqref{limit} or its corresponding Fokker-Planck equation \eqref{vorticity} in some appropriate sense, then we can show that the point vortex system \eqref{particle} is effective in mathematically and physically modeling or numerically simulating the 2D Navier-Stokes equation. Moreover, if we can quantify the mean-field convergence rate or establish the convergence in a better sense, the point vortex system is shown more powerful. Also, the mean-field convergence is also a key tool for understanding the large point vortex system when $N$ becomes extremely large. In such cases, it is difficult to track and analyze the trajectory and the behavior of each particle. Therefore, we turn to the statistical behavior of the particles, which can be well approximated by the limit density $\omega_t$. The arguments above also apply to more general models, including many stochastic interacting particle systems in higher dimensions.

The relation between the particle system \eqref{particle} and the limit McKean-Vlasov system \eqref{limit} as $N \rightarrow \infty$ lies in the regime of the \textit{mean-field theory}. It is expected to justify that, if the initial empirical measure $\frac{1}{N} \sum_{i=1}^N \M_i \delta_{X_i(0)}$ converges weakly to $\omega_0$, then for any $t>0$, the empirical measure $\frac{1}{N} \sum_{i=1}^N \M_i \delta_{X_i(t)}$ of the particle system \eqref{particle} also converges weakly to $\omega_t$, where $\omega_t$ solves the 2D Navier-Stokes equation \eqref{vorticity}. In the language of Kac's molecular chaos \cite{kac1956foundations}, this is called \textit{propagation of chaos}. An equivalent argument can be stated in the PDE context. For any fixed $k$, denote by $G_t^{N,k}$ the probability density function of $(\M_1, X_1, \cdots, \M_k, X_k)$ at time $t$ and let $\omega_t^{N,k}=\E_{\M_1 \cdots \M_k} G_t^{N,k}$ be the $k$-particle vorticity function. We expect that if $\omega_0^{N,k}$ converges weakly to $\omega_0^{\otimes k}$ initially, then for any $t>0$, $\omega_t^{N,k}$ also converges weakly to $\omega_t^{\otimes k}$. This argument shows that, although the particles cannot be independent when $t>0$ due to the pairwise interaction, any given fixed $k$ particles are asymptotically independent as $N \to \infty$, recovering the independence in some sense. Moreover, this convergence can be quantified using the distances between probability density functions, which is referred to \textit{quantitative propagation of chaos}. For detailed information and comprehensive results of recent developments in mean-field theory and propagation of chaos, we refer to reviews such as \cite{sznitman1991topics,jabin2014review,golse2016dynamics,jabin2017mean}.

For our stochastic point vortex system approximating the limit 2D Navier-Stokes equation, the first such result dates back to Osada \cite{osada1986propagation}, under the assumption of large viscosity. Propagation of chaos in the $L^1$ sense with positive viscosity was first established in the seminal paper of Fournier-Hauray-Mischler \cite{fournier2014propagation}, where they studied the entropy and Fisher information of the point vortex system and used a compactness argument, but without any explicit convergence rate. This article is devoted to quantitative propagation of chaos for the point vortex system \eqref{particle}. 

\begin{rmk}\label{fixedcirculation}
    A special case of our stochastic point vortex system \eqref{particle}, say the fixed circulations case $\M_i \equiv 1$, has been well-studied for approximating the 2D Navier-Stokes equation \eqref{vorticity}, see for instance the previous quantitative propagation of chaos results \cite{jabin2018quantitative,guillin2024uniform,feng2023quantitative,wang2025sharp}. However, the simplified model is restricted to the case that all circulations have the same sign and magnitude, thus cannot deal with the special vorticity that $\int \omega_t \ud x=0$ for example. This is also reflected by the parabolic maximum principle: $$\inf \omega_t \geq \inf \omega_0=\inf g_0 \geq 0,$$ which shows that the vorticity cannot change sign. We view it as an essential limitation for physical considerations, where one may generally face the $L^1$ vorticity function, while in this article we resolve this difficulty by allowing the circulations to be positive at some points and negative elsewhere.
\end{rmk}

\subsection{Liouville Equation, BBGKY Hierarchy and Limit Equation}

In this subsection, we present the corresponding Liouville (master) equation (forward Kolmogorov equation) and the BBGKY hierarchy of the SDE system \eqref{particle} and the nonlinear Fokker-Planck equation of the limit system \eqref{limit}, since our method will always work on the PDE level. Hereinafter we keep the following notations:
\begin{align*}
    &\M^N=(m_1,\cdots,m_N) \in \R^N, \quad X^N=(x_1,\cdots,x_N) \in \R^{2N},\\
    &Z^N=(\M^N,X^N) \in \mathbb{D}^N=(\R \times \R^2)^N, \quad z=(m,x) \in \mathbb{D}=\R \times \R^2.
\end{align*}
Denote by $G_t^N(\M^N,X^N)$ the joint distribution of $\M^N$ and $X^N$ for the particle system at time $t$, with the factorized initial data:
\begin{equation*}
    G_0^N(Z^N)=g_0^{\otimes N}(Z^N)=\prod_{i=1}^N g_0(z_i),
\end{equation*}
which corresponds to the choice of i.i.d. initial data of the point vortex system \eqref{particle}: $\{ \mathcal{M}_i,X_i(0) \}_{1 \leq i \leq N}$ are $N$ independent copies of $(\mathcal{M},X_0)$ with law $g_0$. Under such initial distribution, the particles in the point vortex system \eqref{particle} are initially well-prepared or well-distributed with probability one, therefore creating no collisions when the system evolves, as proved by Osada \cite{osada1986propagation}. We also define the $k$-marginal of the joint law of the particle system by
\begin{equation*}
    G_t^{N,k}(z_1,\cdots,z_k)=\int_{\mathbb{D}^{N-k}} G_t^N(Z^N) \ud z_{k+1} \cdots \ud z_N. 
\end{equation*}

We define the normalized relative entropy between the joint distribution of the particle system and the $k$-factorized law of the limit system, which will be the major quantity that we shall estimate.

\begin{defi}[Normalized Relative Entropy]\label{entropy}
    The normalized relative entropy between two probability measures $F_k$ and $f_k$ on $\mathbb{D}^k$ is defined by
    \begin{equation*}
        H_k(F_k|f_k)=\frac{1}{k} \int_{\mathbb{D}^k} F_k \log \frac{F_k}{f_k} \ud Z^k.
    \end{equation*}
\end{defi}

\begin{rmk}\label{initialdata}
    We can relax the factorization assumption on the initial data $G_0^N$. For our purposes below, it is sufficient to assume that the marginals of all orders of the initial distribution $G_0^{N,k}$ are close to their factorized counterparts $g_0^{\otimes k}$ in the sense of the normalized relative entropy in Definition \ref{entropy} as below:
    \begin{equation*}
        H_k(G_0^{N,k}|g_0^{\otimes k}) = \frac{1}{k}  \int_{\mathbb{D}^k} G_0^k \log \frac{G_0^k}{g_0^{\otimes k}} \ud Z^k \lesssim \frac{k}{N^2}.
    \end{equation*}
     We choose the simplest factorization assumption here, which corresponds to that the normalized relative entropy $H_k(G_0^{N,k}|g_0^{\otimes k})$ vanishes, but our results can be easily applied to the above more general case.
\end{rmk}

Applying the It\^o formula to the particle system \eqref{particle}, one concludes that the joint law $G_t^N$ solves the following Liouville equation:
\begin{equation}\label{liouville}
    \partial_t G_t^N +\frac{1}{N} \sum_{i,j=1}^N m_j K(x_i-x_j) \cdot \nabla_{x_i} G_t^N=\sigma \sum_{i=1}^N \Delta_{x_i} G_t^N.
\end{equation}
We emphasize that $G_t^N$ shall be understood as the \textit{entropy solution} of the Liouville equation \eqref{liouville}. The definition of entropy solution is given as in Jabin-Wang \cite{jabin2018quantitative}. For the sake of completeness, we state it as follows.
\begin{defi}[Entropy Solution]\label{entropysolution}
    A probability density function $G_t^N \in L^1(\mathbb{D}^N)$ with $G_t^N \geq 0$ and $\int_{\mathbb{D}^N} G_t^N \ud Z^N=1$ is called an entropy solution to the Liouville equation \eqref{liouville} in $[0,T]$, if $G_t^N$ solves \eqref{liouville} in the sense of distribution and for a.e. $t \in [0,T]$ it holds that
    \begin{equation*}
    \begin{aligned}
        & \, \int_{\mathbb{D}^N} G_t^N \log G_t^N \ud Z^N +\sigma \sum_{i=1}^N \int_0^t \int_{\mathbb{D}^N} \frac{|\nabla_{x_i} G_s^N|^2}{G_s^N} \ud Z^N \ud s \\
        \leq &\,  \int_{\mathbb{D}^N} G_0^N \log G_0^N \ud Z^N.
    \end{aligned}
    \end{equation*}
\end{defi}
For the existence of such entropy solution, we also refer to \cite{jabin2018quantitative}.

Using the Liouville equation \eqref{liouville} and the symmetry property of $G_t^N$ with respect to $Z_i$, one can easily deduce that the $k$-marginal $G_t^{N,k}$ solves the following PDE hierarchy:
\begin{equation}\label{hierarchy}
\begin{aligned}
    \partial_t G_t^{N,k}&+\frac{1}{N}\sum_{i,j=1}^k m_j K(x_i-x_j) \cdot \nabla_{x_i} G_t^{N,k}\\
    &+\frac{N-k}{N} \sum_{i=1}^k \int_{\mathbb{D}} m_{k+1}K(x_i-x_{k+1}) \cdot \nabla_{x_i} G_t^{N,k+1} \ud z_{k+1}=\sum_{i=1}^k \Delta_{x_i} G_t^{N,k}.
\end{aligned}
\end{equation}
Hereinafter, we assume that $1 \leq k \leq N$ and use the convention that $G_t^{N,N+1}=0$. The above equation \eqref{hierarchy} describes the evolution of the $k$-marginal using the next order $(k+1)$-marginal, hence forms a hierarchy of PDEs, which is indeed the BBGKY hierarchy of the particle system.

As for the limit system \eqref{limit}, by applying the It\^o formula, we also obtain that the limit density $g_t(m,x)$ solves the following nonlinear Fokker-Planck equation
\begin{equation}\label{fp}
    \partial_t g_t+(K \ast \omega_t) \cdot \nabla_x g_t=\sigma \Delta_x g_t.
\end{equation}
After direct and simple computations, one can verify that the $k$-factorized law of $g_t(m,x)$, denoted by $g_t^k(\M^k,X^k)$, solves the following PDE:
\begin{equation}\label{tensorize}
    \partial_t g_t^k+\sum_{i=1}^k (K \ast \omega_t)(x_i) \cdot \nabla_{x_i} g_t^k=\sigma \sum_{i=1}^k \Delta_{x_i} g_t^k.
\end{equation}

Hereinafter, we generally drop the index $t$ for simplicity. Also, since the variable $\M$ is independent of time $t$, if we let $f^m(x)=g(m,x)=\mathbb{P}(X|\M=m)$ be the conditional probability density of $X$ under $\M=m$ as in \cite{fournier2014propagation}, we see that $f^m$ satisfies the same evolution equation \eqref{fp} with $g$. Therefore, we shall not make the difference between $g(m,x)$ and $f^m(x)$ for calculations, and we assume that all derivatives of $g$ without index only act on the $x$ variable. In other words, we view $m$ as a fixed parameter.

\subsection{Main Results}

The main result of this article is to prove quantitative propagation of chaos for the point vortex system \eqref{particle} on the whole space. We shall both work in the global and local sense. In this article, global convergence is termed as the asymptotic vanishing of $H_N(G^N \vert g^N)$, while local convergence is for the asymptotic vanishing of $H_k(G^{N, k} \vert g^k)$. 

For the global case, inspired by the groundbreaking paper \cite{jabin2018quantitative}, we consider the global normalized relative entropy between the joint law of the particle system and the factorized law of the limit system, which controls the $L^1$ distance of marginals between two systems by the Csisz\'ar-Kullback-Pinsker (CKP) inequality. By estimating the time evolution of the relative entropy, using the Law of Large Numbers and Large Deviation type results introduced by \cite{jabin2018quantitative}, we are able to close the Gr\"onwall loop and arrive at a global optimal decay rate, hence proving the strong propagation of chaos result. We mention that in this case we prove our results for all positive viscosity $\sigma>0$. 

\begin{thm}[Global Entropic Propagation of Chaos]\label{entropic}
    Assume that $G^N \in \mathcal{P}(\mathbb{D}^N)$ is an entropy solution of the Liouville equation \eqref{liouville} and that $g \in L^\infty([0,T];W^{2,1} \cap W^{2,\infty}(\mathbb{D}))$ solves \eqref{fp} with $g \geq 0$ and $\int_{\mathbb{D}} g(z) \ud z=1$. Assume further that   the initial data  $g_0 \in W_x^{2,1} \cap W_x^{2,\infty}(\mathbb{D})$ satisfies the logarithmic growth conditions
    \begin{equation}
        |\nabla \log g_0(m,x)| \lesssim 1+|x|
    \end{equation}
    \begin{equation}
        |\nabla^2 \log g_0(m,x)| \lesssim 1+|x|^2
    \end{equation}
    and the Gaussian upper bound
    \begin{equation}
        g_0(m,x) \leq C_0 \exp(-C_0^{-1} |x|^2)
    \end{equation}
    for some $C_0>0$. Then we have
    \begin{equation}
        H_N(G^N|g^N) \leq M(1+t)^M \Big(H_N(G_0^N|g_0^N)+\frac{1}{N}\Big)
    \end{equation}
    for any $t \in [0,T]$, where $M$ is some universal constant depending only on $\sigma, A,C_0$ and some Sobolev norms and logarithmic bounds of the initial data $g_0$. 
\end{thm}
Applying the CKP inequality and using that $A \leq 1$ , we establish the global $L^1$ strong propagation of chaos towards the solution to the vorticity formulation of the 2D Navier-Stokes equation \eqref{vorticity}.
\begin{cor}[Propagation of Chaos in $L^1$]\label{global}
    Under the same assumption as Theorem \ref{entropic}, define the $k$-particle vorticity function $\omega^{N,k}$ by
    \begin{equation*}
        \omega^{N,k}(x_1,\cdots,x_k)=\int_{\R^k \times \mathbb{D}^{N-k}} m_1 \cdots m_k G^N(Z^N) \ud m_1 \cdots \ud m_k \ud z_{k+1} \cdots\ud z_N.
    \end{equation*}
    Then we have
    \begin{equation*}
        \| \omega^{N,k}-\omega^{\otimes k} \|_{L^1(\R^{2k})} \leq \| G^{N,k}-g^{\otimes k} \|_{L^1(\mathbb{D}^{k})}\leq \frac{\sqrt{k}}{\sqrt{N}}M(1+t)^M
    \end{equation*}
    for any $t \in [0,T]$ and some universal constant $M$.
\end{cor}

For the local case, as observed first by Lacker \cite{lacker2021hierarchies}, in order to obtain the optimal convergence rate in $N$ between $G^{N,k}$ and $g^k$ for any fixed $k$, we shall consider the local normalized relative entropy directly instead of using the sub-additivity property
\begin{equation*}
    H_k(G^{N,k}|g^k) \leq H_N(G^N|g^N)
\end{equation*}
to derive the local estimates from the global ones. Namely, this will provide the convergence rate in $L^1$ distance as
\begin{equation*}
    \|\omega^{N,k}-\omega^{\otimes k} \|_{L^1(\R^{2k})}\lesssim \frac{\sqrt{k}}{\sqrt{N}},
\end{equation*}
as shown in Corollary \ref{global} above. However, the best optimal convergence rate one may expect is actually $O(k/N)$, tested by some simple Gaussian examples. Using the BBGKY hierarchy to construct an ODE hierarchy of relative entropy $H_k$ of all orders, this optimal rate has been verified for Lipschitz kernels in \cite{lacker2021hierarchies,lacker2023sharp}. The recent work by S. Wang \cite{wang2025sharp} has extended this idea to the 2D Navier-Stokes equation with fixed circulations $\M_i \equiv 1$ on the torus, being the first attempt for solving the singular kernel case. Our results shall improve it to the general circulations and to the whole space, but we mention that our result is still restricted to the high viscosity case, say large enough $\sigma>0$, as in \cite{wang2025sharp}. 

\begin{thm}[Sharp Local Entropic Propagation of Chaos]\label{localentropic}
    Under the same assumption as Theorem \ref{entropic}, we further assume that the viscosity constant $\sigma$ is large enough in the sense of
    \begin{equation*}
        \sigma >\sqrt{2}A\|V\|_{L^\infty},
    \end{equation*}
    with $V(x)=\Big(-\frac{1}{2\pi} \arctan \frac{x_1}{x_2} \Big) \mathrm{Id}$ satisfying $K=\nabla \cdot V$. Then we have
    \begin{equation}
        H_k(G^{N,k}|g^k) \leq M(1+t)^M \Big(H_k(G_0^{N,k}|g_0^k)+\frac{k}{N^2}\Big)
    \end{equation}
    for any $t \in [0,T]$, where $M$ is some universal constant depending only on $\sigma, A,C_0$ and some Sobolev norms and logarithmic bounds of the initial data $g_0$. 
\end{thm}
Applying again the CKP inequality and using that $A\leq 1$ , we establish the sharp local $L^1$ strong propagation of chaos towards the solution to the vorticity formulation of the 2D Navier-Stokes equation \eqref{vorticity}.
\begin{cor}[Sharp Propagation of Chaos in $L^1$]\label{local}
    Under the same assumption as Theorem \ref{localentropic}, we have
    \begin{equation*}
        \| \omega^{N,k}-\omega^{\otimes k} \|_{L^1(\R^{2k})} \leq \frac{k}{N}M(1+t)^M
    \end{equation*}
    for any $t \in [0,T]$ and some universal constant $M$.
\end{cor}

\begin{rmk}
    We should state here why we cannot assume that $\M_i$ and $X_i(0)$ are mutually independent in our vortex model \eqref{particle}. This is essentially required when studying propagation of chaos for this system, since we need to impose Kac's chaos condition on the initial data, say
    \begin{equation*}
        \frac{1}{N} \sum_{i=1}^N \M_i \delta_{X_i(0)} \rightharpoonup \omega_0
    \end{equation*}
    as $N \rightarrow \infty$. Let $\varphi$ be an arbitrary test function, then we have by independence (if made by assumption)
    \begin{equation*}
        \E \Big\langle \frac{1}{N} \sum_{i=1}^N \M_i \delta_{X_i(0)},\varphi \Big\rangle=\E \M \cdot \E \varphi(X_0) \rightarrow \int_{\R^2} \varphi(x) \, \omega_0(x) \ud x.
    \end{equation*}
    Hence we must have
    \begin{equation*}
        \omega_0(x)=\E \M \cdot f_0(x),
    \end{equation*}
    where $f_0(x)$ is the probability density of $X(0)$. This shows that $\omega_0(x)$, and hence $\omega_t(x)$ by the parabolic maximum principle, must be always non-positive or non-negative, depending on whether $\E \M$ is non-positive or not. Hence we will not be able to approximate general vorticity which may change sign. This will eliminate the physical meaning of our study, as we expect that our results can provide particle approximation arguments to general vorticity functions $\omega_t$ by constructing the system with general circulations. This coincides with our previous remark about the significance of the appearance of the circulation $\M_j$ in the point vortex system \eqref{particle}. One may encounter the same restriction if the initial density of the limit system is assumed to be separating variables, say
    \begin{equation*}
        g_0(m,x)=r_0(m)f_0(x),
    \end{equation*}
    although this will provide convenience in mathematical computations. Therefore, we also cannot make this assumption as in \cite{shao2024quantitative}. 
    
    However, we also emphasize that the independence assumption is not necessary to deal with the torus case in \cite{shao2024quantitative}, although they did not point it out in their paper. For completeness, we will give a short sketch of the proof of the torus case without the independence assumption in our proof of main results, but this result is essentially due to \cite{shao2024quantitative} and we shall claim no originality to this.
\end{rmk}

\subsection{Discussion on the Model}
In this subsection, we gather some discussions on the vortex model \eqref{particle} in our main results, concerning the main motivation to include the general circulations and the whole space as a state space and also the possible limitations on bounded domains. We emphasize that these conditions are motivated both physically and mathematically.

The appearance of general circulations $\M_i/N$ essentially improves the universality of the vortex model, making it capable of approximating more general fluid flows. In the realistic setting, vortex models such as a dipole naturally contain vortices of different signs. Also in the real-world flows, it is natural to expect that the strength and the direction of the local rotation in the fluid are not constant everywhere. Therefore, such cases require a point vortex model with different circulations for a suitable approximation. Mathematically, as explained in Remark \ref{fixedcirculation}, the simpler model with fixed circulations $\M_i \equiv 1$, although being easier for computation, requires that the vorticity $\omega_t$ is given by a probability density function , or at least a fixed-sign function. This limitation prevents the approximation of physically more general and realistic vorticity functions.

We have also chosen the whole space $\R^2$ as a state space for the point vortex model, instead of a compact manifold like the torus on which many previous studies worked. Namely, we assume on the position of the vortices $X_i \in \R^2$ rather than $X_i \in \mathbb{T}^2$. The physical consideration is also to restore the generality of our model. Many real fluids, like the ocean or a large body of water, exists in extremely large domains. It is somehow artificial and unrealistic to constrain them in a compact manifold. Moreover, confining the fluid and the model to a torus imposes an extra periodicity which is non-physical, and also prevents us from studying how the vorticity spreads and decays at infinity. It also requires new mathematical techniques to deal with the vorticity function on the whole space, namely the sharp logarithmic growth estimates in our results, in order to handle the quantities like $\nabla \log g$ which are typically unbounded on the whole space but are easier to estimate on the torus. This provides new insights for the study of the singular 2D Biot-Savart kernel and the similar behavior of the fundamental solution with the well-known Oseen vortices.

There are possible limitations when we work on the realistic bounded domains. On a bounded domain, the Navier-Stokes equation requires physical boundary conditions, most commonly the \textit{no-slip condition} (the fluid velocity is zero at the boundary). But the vorticity formulation does not behave well with these boundary conditions. In a real viscous fluid, the no-slip condition at the boundary provides a primary mechanism for \textit{vorticity generation}. Vorticity is created at the boundary and diffuses into the flow. Moreover, the boundary condition may also be not compatible with the mean-field scaling. The entire derivation of the mean-field limit relies on the fact that the interaction force on a given vortex is a sum of $N-1$ terms, each of order $1/N$. On a bounded domain, however, the interaction with the boundary does not necessarily respect this scaling. The influence of the boundary may become a dominant non-mean-field effect.

\subsection{Related Works}

Propagation of chaos and convergence to mean-field limit for large systems of interacting particles, widely appearing in the analysis of PDEs and probability theory, have been extensively studied over the last decades, both for the first-order and second-order systems. The derivation of some effective equations to describe the average or statistical behavior of large interacting particle systems can be traced back to Maxwell \cite{Maxwell1867} and Boltzmann \cite{boltzmann1872weitere}. Kac \cite{kac1956foundations} first raised the basic notions of chaos and proved some partial propagation of chaos results for the simplified model of the Boltzmann equation in dimension $1$. Related studies have been far beyond the range of kinetic theory. The fundamental mathematical setting of this problem was set up and basically resolved by McKean \cite{mckean1967propagation} and Dobrushin \cite{dobrushin1979vlasov} using the classical coupling method. Since then people began to pay attention to the much more difficult models, say systems with singular interacting kernels, instead of the regular (bounded Lipschitz) kernels. As for our point vortex model approximating the famous 2D Navier-Stokes/Euler equation, the first result dates back to Osada \cite{osada1986propagation} with bounded initial data and large viscosity. Later Fournier-Hauray-Mischler \cite{fournier2014propagation} extended the result to all positive viscosity cases, while the initial data is only assumed to have finite moments and entropy, using the compactness argument and the entropy estimate.

The relative entropy method as in Jabin-Wang \cite{jabin2016mean,jabin2018quantitative} has seen great effectiveness in proving quantitative propagation of chaos results for McKean-Vlasov processes in recent years, especially making sense for the singular interacting kernel cases. In \cite{jabin2016mean}, Jabin and the second author resolved the second-order case with bounded kernels. In \cite{jabin2018quantitative} they successfully applied to the 2D Navier-Stokes equation with the singular kernel given by the Biot-Savart law, but restricted to the torus case and without general circulations. Later extended results include Wynter \cite{wynter2023quantitative} for the mixed-sign case, Guillin-Le Bris-Monmarch\'e \cite{guillin2024uniform} for uniform-in-time propagation of chaos using the logarithmic Sobolev inequality (LSI) for the limit density, and Shao-Zhao \cite{shao2024quantitative} for general circulations with common environmental noise, all of which are still working on the torus. In our previous work \cite{feng2023quantitative} we extended the result to the whole Euclidean space, given some special growth conditions on the initial data, which are natural and reasonable since those estimates automatically hold for general $L^1$ initial data with Gaussian upper bound as long as the time becomes positive. These arguments appear to be more reasonable when considering the strengthening chaos results for instance by Huang \cite{huang2023entropy}, where the entropic propagation of chaos is derived for any positive time given only initial chaos in the Wasserstein sense, for bounded interacting kernels. We also mention the recent applications of the relative entropy method to establish propagation of chaos for other models, for instance the Landau equation with Maxwellian molecules by Carrillo-Feng-Guo-Jabin-Wang \cite{carrillo2025relative} and some uniform-in-time results for second-order systems with regular enough kernels by Gong-Wang-Xie \cite{gong2024uniform}.

The modulated energy method introduced by Serfaty \cite{serfaty2020mean} is another effective tool for resolving propagation of chaos for the singular kernel case, especially the Coulomb/Riesz cases. By constructing a Coulomb/Riesz based metric between the empirical measure of the $N$-particle system and the limit density. Serfaty \cite{serfaty2020mean} established the quantitative propagation of chaos for the first-order Coulomb and super-Coulomb cases without noise in the whole space, in the sense of controlling the weak convergence rate when acting on test functions. This method has been extended to various problems, including Rosenzweig \cite{rosenzweig2022mean} for 2D incompressible Euler equation, Nguyen-Rosenzweig-Serfaty \cite{nguyen2022mean} for Riesz-type flows, Rosenzweig-Serfaty \cite{rosenzweig2023global} for uniform-in-time estimates for sub-Coulomb cases, and Chodron De Courcel-Rosenzweig-Serfaty \cite{chodron2023sharp} for periodic Riesz cases. The key estimates appearing in the control of modulated energy, say the commutator type estimates, have been thoroughly examined in the recent work Rosezweig-Serfaty \cite{rosenzweig2025sharp}. The combination of modulated energy and relative entropy leads to the modulated free energy method, which is applied to the Coulomb/Riesz kernels on the torus and more interestingly the first time to the attractive kernel cases by Bresch-Jabin-Wang \cite{bresch2019mean,bresch2019modulated,bresch2023mean} for the 2D log gas and the 2D Patlak-Keller-Segel model. See also the recent work by Chodron De Courcel-Rosenzweig-Serfaty \cite{chodron2025attractive} for further results on uniform-in-time propagation of chaos and phase transition for attractive log gas. In the recent work Cai-Feng-Gong-Wang \cite{cai2024propagation}, the authors and coauthors are also able to apply the modulated free energy method to resolve the 2D Coulomb case on the whole space. Some generation of chaos results are also obtained partially by Rosenzweig-Serfaty \cite{rosenzweig2025modulated,rosenzweig2024relative} using the time-uniform LSI for the limit density or the modulated Gibbs measure. We refer to the recent lecture notes by Serfaty \cite{serfaty2024lectures} for more complete discussions on the Coulomb/Riesz gases concerning above topics.

We also emphasize that in the work of Lacker \cite{lacker2021hierarchies} and Lacker-Le Flem \cite{lacker2023sharp}, it is shown that using the BBGKY hierarchy that the optimal local convergence rate of the relative entropy between $k$-marginals is $O(k/N)$ for systems with bounded interactions, while previous global arguments can only obtain the suboptimal $O(\sqrt{k/N})$ rate. This is verified by calculating the local relative entropy and solving the ODE hierarchy by iteration. Most recently in S. Wang \cite{wang2025sharp} the optimal rate has been extended to the 2D Biot-Savart case with large viscosity on the torus also using  the BBGKY hierarchy. Qualitative propagation of chaos for the 2D Vlasov-Poisson-Fokker-Planck equation has also been established by Bresch-Jabin-Soler \cite{bresch2025new} by hierarchy estimates.  Recently, a new duality based method for the mean-field limit problem was introduced by Bresch-Duerinckx-Jabin \cite{bresch2024duality}, which requires only that the interacting kernel is in $L^2$.

\subsection{Outline of the Article}

The rest of this paper is organized as follows. We first present the proof of our main results in Section \ref{main}, acknowledging some logarithmic growth estimates of the limit density. The proof is divided into several subsections, respectively giving the time evolution of the relative entropy, some useful functional inequality lemmas and theorems, and the main controlling arguments. We shall provide the global and local relative entropy control respectively, and also set up an independent subsection for the proof of an ODE hierarchy proposition. In Section \ref{regularity} we apply the results of Carlen-Loss and Osada and the parabolic maximum principle to give the needed regularity estimates for the limit density. In Section \ref{log} we first present the Grigor'yan parabolic maximum principle as the key tool, and then provide our desired logarithmic growth estimates, i.e. growth controls on $\nabla \log g$ and $\nabla^2 \log g$ under certain initial conditions.

\section{Proof of Main Results}\label{main}

In this section, we apply the relative entropy method to prove the propagation of chaos main results with respect to the above systems \eqref{particle} and \eqref{limit}. Due to the partial similarity of the methods and mid-steps between the global and local cases, we shall present the proofs simultaneously.

\subsection{Time Evolution of Relative Entropy}\label{time}

We present in this subsection the time evolution of the normalized relative entropy, which is computed directly using the evolution PDE of the two systems. First, we give the computations for the global relative entropy.
\begin{lemma}[Evolution of Global Relative Entropy]\label{evolution}
    \begin{align*}
        \frac{\ud}{\ud t} H_N(G^N|g^N)=&-\frac{\sigma}{N} \sum_{i=1}^N \int_{\mathbb{D}^N} G^N \Big|\nabla_{x_i} \log \frac{G^N}{g^N}\Big|^2 \ud Z^N\\
        &-\frac{1}{N^2} \sum_{i,j=1}^N \int_{\mathbb{D}^N} G^N \Big(m_j K(x_i-x_j)-K \ast \omega (x_i)\Big) \cdot \nabla \log g(m_i,x_i) \ud Z^N.
    \end{align*}
\end{lemma}

\begin{proof}
    Recall that $G^N$ satisfies the Liouville equation \eqref{liouville}:
    \begin{equation*}
        \partial_t G^N +\frac{1}{N} \sum_{i,j=1}^N m_j K(x_i-x_j) \cdot \nabla_{x_i} G^N=\sigma \sum_{i=1}^N \Delta_{x_i} G^N,
    \end{equation*}
    while the factorzied law $g^N$ solves \eqref{tensorize}:
    \begin{equation*}
        \partial_t g^N+\sum_{i=1}^N K \ast \omega(x_i) \cdot \nabla_{x_i} g^N=\sigma \sum_{i=1}^N \Delta_{x_i} g^N.
    \end{equation*}
    Hence by direct computations we have
    \begin{align*}
        &\frac{\ud}{\ud t} H_N(G^N|g^N)\\
        =\,&\frac{1}{N} \int_{\mathbb{D}^N} \partial_t G^N(\log G^N+1-\log g^N)+\frac{G^N}{g^N} \partial_t g^N \ud Z^N\\
        =\,&\frac{1}{N}\int_{\mathbb{D}^N} \Big(\sigma \sum_{i=1}^N \Delta_{x_i} G^N-\frac{1}{N} \sum_{i,j=1}^N m_j K(x_i-x_j) \cdot \nabla_{x_i} G^N\Big)(\log G^N+1-\log g^N) \ud Z^N\\
        &+\frac{1}{N} \int_{\mathbb{D}^N} \Big(\sigma \sum_{i=1}^N \Delta_{x_i} g^N-\sum_{i=1}^N K \ast \omega(x_i) \cdot \nabla_{x_i} g^N\Big)\frac{G^N}{g^N} \ud Z^N.
    \end{align*}
    Using integration by parts, we arrive at
    \begin{align*}
        &-\frac{\sigma}{N} \sum_{i=1}^N \int_{\mathbb{D}^N} \nabla_{x_i} G^N \cdot \nabla_{x_i}(\log G^N-\log g^N) \ud Z^N \\
        &+\frac{1}{N^2} \sum_{i,j=1}^N \int_{\mathbb{D}^N} m_j K(x_i-x_j) \cdot \nabla_{x_i} (\log G^N-\log g^N) G^N \ud Z^N\\
        &-\frac{\sigma}{N} \sum_{i=1}^N \int_{\mathbb{D}^N} \nabla_{x_i} g^N \cdot \nabla_{x_i} \frac{G^N}{g^N} \ud Z^N\\
        &-\frac{1}{N} \sum_{i=1}^N \int_{\mathbb{D}^N} \Big(K \ast \omega(x_i) \cdot \nabla_{x_i} g^N \Big) G^N \ud Z^N.
    \end{align*}
    By rearranging the above terms, we conclude our desired estimates.
\end{proof}

Now we turn to the computations for the local relative entropy.

\begin{lemma}[Evolution of Local Relative Entropy]\label{evolutionlocal}
    \begin{align*}
        &\frac{\ud}{\ud t} H_k(G^{N,k}|g^k)\\
        =&-\frac{\sigma}{k} \sum_{i=1}^k \int_{\mathbb{D}^k} G^{N,k} \Big|\nabla_{x_i} \log \frac{G^{N,k}}{g^k}\Big|^2 \ud Z^k\\
        &-\frac{1}{k} \sum_{i=1}^k \frac{1}{N} \sum_{j=1}^k \int_{\mathbb{D}^k} G^{N,k} \Big(m_j K(x_i-x_j)-K \ast \omega (x_i)\Big) \cdot \nabla \log g(m_i,x_i) \ud Z^k\\
        &-\frac{1}{k}\sum_{i=1}^k \frac{N-k}{N} \int_{\mathbb{D}^{k+1}} G^{N,k+1} \Big(m_{k+1}K(x_i-x_{k+1})-K \ast \omega(x_i) \Big) \cdot \nabla \log g(m_i,x_i) \ud Z^{k+1}\\
        &+\frac{1}{k}\sum_{i=1}^k \frac{N-k}{N} \int_{\mathbb{D}^{k+1}} G^{N,k+1} \Big(m_{k+1}K(x_i-x_{k+1}) \cdot \nabla_{x_i} \log G^{N,k}\Big) \ud Z^{k+1}.
    \end{align*}
\end{lemma}

\begin{proof}
    Recall that $G^{N,k}$ satisfies the BBGKY hierarchy \eqref{hierarchy}:
   \begin{equation*}
\begin{aligned}
    \partial_t G^{N,k}&+\frac{1}{N}\sum_{i,j=1}^k m_j K(x_i-x_j) \cdot \nabla_{x_i} G^{N,k}\\
    &+\frac{N-k}{N} \sum_{i=1}^k \int_{\mathbb{D}} m_{k+1}K(x_i-x_{k+1}) \cdot \nabla_{x_i} G^{N,k+1} \ud z_{k+1}=\sum_{i=1}^k \Delta_{x_i} G^{N,k}.
\end{aligned}
\end{equation*}
    while the factorized law $g^k$ solves \eqref{tensorize}:
    \begin{equation*}
        \partial_t g^k+\sum_{i=1}^k K \ast \omega(x_i) \cdot \nabla_{x_i} g^k=\sigma \sum_{i=1}^k \Delta_{x_i} g^k.
    \end{equation*}
    Hence by direct computations we have
    \begin{align*}
        &\frac{\ud}{\ud t} H_k(G^{N,k}|g^k)\\
        =\,&\frac{1}{k} \int_{\mathbb{D}^k} \partial_t G^{N,k}(\log G^{N,k}+1-\log g^k)+\frac{G^{N,k}}{g^k} \partial_t g^k \ud Z^k\\
        =\,&\frac{1}{k}\int_{\mathbb{D}^k} \Big(\sigma \sum_{i=1}^k \Delta_{x_i} G^{N,k}-\frac{1}{N} \sum_{i,j=1}^k m_j K(x_i-x_j) \cdot \nabla_{x_i} G^{N,k}\\
        &\qquad -\frac{N-k}{N} \sum_{i=1}^k \int_{\mathbb{D}} m_{k+1}K(x_i-x_{k+1}) \cdot \nabla_{x_i} G^{N,k+1} \ud z_{k+1}\Big)(\log G^{N,k}+1-\log g^k) \ud Z^k\\
        &+\frac{1}{k} \int_{\mathbb{D}^k} \Big(\sigma \sum_{i=1}^k \Delta_{x_i} g^k-\sum_{i=1}^k K \ast \omega(x_i) \cdot \nabla_{x_i} g^k\Big)\frac{G^{N,k}}{g^k} \ud Z^k.
    \end{align*}
    Using integration by parts, we arrive at
    \begin{align*}
        &-\frac{\sigma}{k} \sum_{i=1}^k \int_{\mathbb{D}^k} \nabla_{x_i} G^{N,k} \cdot \nabla_{x_i}(\log G^{N,k}-\log g^k) \ud Z^k \\
        &+\frac{1}{k} \sum_{i=1}^k \frac{1}{N} \sum_{j=1}^k \int_{\mathbb{D}^k} m_j K(x_i-x_j) \cdot \nabla_{x_i} (\log G^{N,k}-\log g^k) G^{N,k} \ud Z^k\\
        &+\frac{1}{k}\sum_{i=1}^k \frac{N-k}{N} \int_{\mathbb{D}^{k+1}} m_{k+1} K(x_i-x_{k+1}) \cdot \nabla_{x_i}(\log G^{N,k}-\log g^k) G^{N,k+1} \ud Z^{k+1}\\
        &-\frac{\sigma}{k} \sum_{i=1}^k \int_{\mathbb{D}^k} \nabla_{x_i} g^k \cdot \nabla_{x_i} \frac{G^{N,k}}{g^k} \ud Z^k\\
        &-\frac{1}{k} \sum_{i=1}^k \int_{\mathbb{D}^k} K \ast \omega(x_i) \cdot \nabla_{x_i} g^k G^{N,k} \ud Z^k.
    \end{align*}
    By rearranging the above terms, we conclude our desired estimates.
\end{proof}

\subsection{Law of Large Numbers and Large Deviation Type Theorems}\label{thms}

In the original paper \cite{jabin2018quantitative}, Jabin-Wang introduced two Law of Large Numbers and Large Deviation type theorems, which are crucial in controlling the expectation of some mean-zero functions with respect to the joint law, bounding it with the normalized relative entropy itself and some $O(1/N)$ error terms, while one may only get $O(1)$ bound by crude estimates. Their method includes expanding the exponential term by Taylor's expansion and counting carefully by combinatorial techniques that how many terms actually do not vanish, using the one-side or two-side cancellation rules. Note that in our setting, since we should view $(m,x)$ together as a single variable $z$ to keep the symmetry, the two theorems should be changed formally, see also \cite{shao2024quantitative}.

\begin{thm}[Law of Large Numbers]\label{lln}
    Consider any probability density $\bar\rho(z)$ on $\mathbb{D}$ and any test function $\phi(z,w) \in L^\infty(\mathbb{D}^2)$ satisfying the one-side cancelling property
    \begin{equation}\label{cancel1}
        \int_{\mathbb{D}} \phi(z,w)\bar\rho(w) \ud w=0, \, \mbox{for any\, } z,
    \end{equation}
    and $\| \phi \|_{L^\infty}<\frac{1}{2e}$.
    Then we have
    \begin{equation*}
        \int_{\mathbb{D}^N} \bar\rho^{\otimes N} \exp\Big(\frac{1}{N}\sum_{j,k=1}^N \phi(z_1,z_j)\phi(z_1,z_k)\Big) \ud Z^N \leq C<\infty.
    \end{equation*}
\end{thm}
\begin{thm}[Large Deviation]\label{deviation}
    Consider any probability density $\bar\rho(z)$ on $\mathbb{D}$ and any test function $\phi(z,w)$ on $\mathbb{D}^2$ satisfying the two-side cancelling properties
    \begin{equation}\label{cancel2}
        \int_{\R \times \mathbb{R}^2}  \phi(z,w) \bar{\rho}(z)  \ud z=0, \,  \mbox{for any\,  } w, \quad \int_{\R \times \mathbb{R}^2}  \phi(z,w) \bar{\rho}(w) \ud w=0,  \, \mbox{for any\,  } z,
    \end{equation}
    and
    \begin{equation}\label{gamma}
        \gamma=C_{JW} \Big(\sup_{p \geq 1} \frac{\Vert \sup_w |\phi(\cdot, w)| \Vert_{L^p(\bar{\rho} \ud z)}}{p} \Big)^2 <1,
    \end{equation}
    where the constant $C_{JW}=1600^2+36e^4$ is obtained in \cite[Theorem 4]{jabin2018quantitative}. Then we have
    \begin{equation*}
        \int_{(\R \times \mathbb{R}^2)^N}  \bar{\rho}^{\otimes N} \exp\Big(\frac{1}{N} \sum_{i,j=1}^N \phi(z_i,z_j)\Big) \ud Z^N \leq C <\infty.
    \end{equation*}
\end{thm}
We refer the proofs of the above theorems to \cite{jabin2018quantitative}.

The Law of Large Numbers and Large Deviation Theorems above have been successfully applied to estimate the global relative entropy on the torus for various models. But in the setting of this article, there are two main difficulties to overcome. 

The first problem is that we are now working on the whole space, hence we can no longer expect that the test function $\phi$ in Theorem \ref{lln} is bounded. We also mention that Lim-Lu-Nolen \cite{lim2020quantitative} provides another proof for Theorem \ref{deviation} using martingale inequalities when $\phi \in L^\infty$. In the recent work Du-Li \cite{du2024collision} dealing with a new collision-oriented particle system for approximating the Landau type equations, the Law of Large Numbers Theorem \ref{lln} is extended to the case that $\phi$ has linear growth in the second variable $w$ uniformly in the first variable $z$, using the classical Hoeffding's inequality in high-dimensional probability, see for instance the classic book by Vershynin \cite{vershynin2018high}. A special case of this extension is also proved in the recent work of Gong-Wang-Xie \cite{gong2024uniform}. This result should be crucial to our method.

The second problem is that in order to obtain the local optimal rate $O(k^2/N^2)$ for the relative entropy estimate, we need to scale the small constant $\gamma$ by a factor of $k/N$, see Subsection \ref{control2} below. However, the upper bound for the functional first obtained in \cite{jabin2018quantitative} can not well reflect this scaling, so we need to estimate the upper bound more carefully. This has been observed by S. Wang \cite{wang2025sharp} for Theorem \ref{deviation}, and we are now giving similar improvements for both theorems.

\begin{thm}[Law of Large Numbers, an extended version]\label{lln2}
    Consider any probability density $\bar\rho(z)$ on $\mathbb{D}$ and any test function $\phi(z,w)$ on $\mathbb{D}^2$ vanishing on the diagonal $\Delta=\lbrace (z,z): z \in \mathbb{D} \rbrace$ and satisfying the one-side cancelling property
    \begin{equation}\label{cancel3}
        \int_{\mathbb{D}} \phi(z,w)\bar\rho(w) \ud w=0, \, \mbox{for any\, } z,
    \end{equation}
    and for the universal constant $c_0>0$ in the Hoeffding's inequality we have
    \begin{equation*}
        \inf \Big\lbrace c>0: \int_{\mathbb{D}} \exp\Big(\sup_z |\phi(z,w)|^2/c^2\Big) g(w) \ud w \Big\rbrace <\lambda c_0,
    \end{equation*}
    where $\lambda \in (0,\frac{1}{2})$ is some small constant. Then we have
    \begin{equation*}
        \log \int_{\mathbb{D}^N} \bar\rho^{\otimes N} \exp\Big(\frac{1}{N}\sum_{j,k=1}^N \phi(z_1,z_j)\phi(z_1,z_k)\Big) \ud Z^N \leq \frac{C}{\lambda}<\infty.
    \end{equation*}
\end{thm}

\begin{thm}[Large Deviation, an extended version]\label{deviation2}
    Consider any probability density $\bar\rho(z)$ on $\mathbb{D}$ and any test function $\phi(z,w)$ on $\mathbb{D}^2$ vanishing on the diagonal $\Delta=\lbrace (z,z): z \in \mathbb{D} \rbrace$ and satisfying the two-side cancelling properties
    \begin{equation}\label{cancel4}
        \int_{\R \times \mathbb{R}^2}  \phi(z,w) \bar{\rho}(z)  \ud z=0, \,  \mbox{for any\,  } w, \quad \int_{\R \times \mathbb{R}^2}  \phi(z,w) \bar{\rho}(w) \ud w=0,  \, \mbox{for any\,  } z,
    \end{equation}
    and
    \begin{equation}\label{gamma2}
        \gamma=C_{JW} \Big(\sup_{p \geq 1} \frac{\Vert \sup_w |\phi(\cdot, w)| \Vert_{L^p(\bar{\rho} \ud z)}}{p} \Big)^2 <\frac{1}{2}
    \end{equation}
    where the constant $C_{JW}=1600^2+36e^4$ is obtained in \cite[Theorem 4]{jabin2018quantitative}. Then we have
    \begin{equation*}
        \log \int_{(\R \times \mathbb{R}^2)^N}  \bar{\rho}^{\otimes N} \exp\Big(\frac{1}{N} \sum_{i,j=1}^N \phi(z_i,z_j)\Big) \ud Z^N \leq C\gamma <\infty.
    \end{equation*}
\end{thm}

The proof of the extended Law of Large Numbers Theorem \ref{lln2} is referred to \cite{du2024collision}, where the constant upper bound is slightly modified, while the proof of the extended Large Deviation Theorem \ref{deviation2} is referred to \cite{wang2025sharp} without any changes.

Finally, we present two technical lemmas that shall be used in our proof. The first one is the famous Donsker-Varadhan inequality, which comes from a variational formula of the relative entropy. This inequality allows us to control the expectation of a test function with respect to a non-factorized density $\rho_N$ by the expectation of another exponential test function with respect to a factorized density $\bar{\rho}^{\otimes N}$ and the relative entropy between the two densities. The expectation with respect to a factorized density can be estimated by the Law of Large Numbers Theorem \ref{lln2} and the Large Deviation Theorem \ref{deviation2} above.

\begin{lemma}[Donsker-Varadhan Inequality]\label{gibbs}
    Consider any probability density $\bar\rho(z)$ on $\mathbb{D}$ and $\rho_N(Z^N)$ on $\mathbb{D}^N$ and any test function $\Phi(Z^N)$ on $\mathbb{D}^N$, then for any $\eta>0$ we have
    \begin{equation*}
        \int_{\mathbb{D}^N} \Phi\rho_N \ud Z^N \leq \frac{1}{\eta}H_N(\rho_N|\bar\rho^{\otimes N})+\frac{1}{\eta N}\log \int_{\mathbb{D}^N} \exp(N \eta \Phi) \bar{\rho}^{\otimes N} \ud Z^N.
    \end{equation*}
\end{lemma}

The proof of Lemma \ref{gibbs} is rather standard in the literature. We refer to \cite[Lemma 1]{jabin2018quantitative} for instance. The second one is the chain rule of the relative entropy between two disintegrated probability measures, which allows us to estimate the relative entropy between two conditional measures.

\begin{lemma}[Chain Rule of Relative Entropy]\label{chain}
    For any two disintegrated probability measures $(m^i(\ud x) K_x^i(\ud y))_{i=1,2}$ on a product space $E_1 \times E_2$, we have
    \begin{equation*}
    \begin{aligned}
        \int_{E_1} H(K_x^1|K_x^2) m^1(\ud x) &\, \leq H(m_1|m_2)+\int_{E_1} H(K_x^1|K_x^2) m^1(\ud x)\\
        &\, =  H\Big(m^1(\ud x) K_x^1(\ud y)|m^2(\ud x) K_x^2(\ud y)\Big).
    \end{aligned}
    \end{equation*}
\end{lemma}

We refer the proof of Lemma \ref{chain} to Budhiraja-Dupuis \cite{budhiraja2019analysis}.

\subsection{Global Relative Entropy Control}\label{control}

In this subsection, with all the key tools for preparation in the last subsection, we shall prove our main estimate on the right-hand side of Lemma \ref{evolution}. Thanks to the non-positive dissipation term, we only need to control the second term, which we recall
\begin{equation*}
    -\frac{1}{N^2} \sum_{i,j=1}^N \int_{\mathbb{D}^N} G^N \Big(m_j K(x_i-x_j)-K \ast \omega (x_i)\Big) \cdot \nabla \log g(m_i,x_i) \ud Z^N.
\end{equation*}
Due to the appearance of $m_j$ in front of the singular kernel $K(x_i-x_j)$, it seems impossible to use the symmetrization trick as in \cite{jabin2018quantitative,feng2023quantitative}. Hence we use integration by parts, thanks to that $K \in W^{-1,\infty}$, with $K=\nabla \cdot V$ as in \cite{jabin2018quantitative} where
\begin{equation*}
    V(x)=\Big(-\frac{1}{2\pi}\arctan\frac{x_1}{x_2}\Big)\text{Id}
\end{equation*}
is bounded, to eliminate the singularity and obtain
\begin{align}
    &\, \frac{1}{N^2} \sum_{i,j=1}^N \int_{\mathbb{D}^N} \Big(m_jV(x_i-x_j)-V \ast \omega(x_i)\Big):\Big(\nabla_{x_i}g^N \otimes \nabla_{x_i}\frac{G^N}{g^N}\Big) \ud Z^N\label{thm1}\\
    +&\, \frac{1}{N^2} \sum_{i,j=1}^N \int_{\mathbb{D}^N} G^N\Big(m_jV(x_i-x_j)-V\ast \omega(x_i)\Big):\frac{\nabla_{x_i}^2 g^N}{g^N} \ud Z^N.\label{thm2}
\end{align}

\textbf{Control of \eqref{thm2} :} For the second term \eqref{thm2}, we denote that the test function
\begin{equation*}
    \phi(z,w)=\phi(m,x,m^\prime,x^\prime)=\Big(m^\prime V(x-x^\prime)-V \ast \omega(x)\Big):\frac{\nabla^2 g}{g}(m,x).
\end{equation*}
It is straightforward to check that $\phi$ satisfies the two cancelling properties \eqref{cancel2}. Also, we have $\phi$ satisfies the growth condition
\begin{equation*}
    f(m,x) \triangleq \sup_{m^\prime,x^\prime} |\phi(m,x,m^\prime,x^\prime)| \leq C\frac{B_1}{1+t}(B_2-\log(1+t)-\log g(m,x))
\end{equation*}
since $V$ is bounded and $\M$ is compactly supported and $\nabla^2 g/g$ satisfies the logarithmic estimates as in Proposition \ref{logarithmic} below. Moreover, for any $\lambda(t)>0$, as in \cite{jabin2016mean} we have
\begin{equation*}
    \int_{\mathbb{D}} e^{\lambda(t)f} g(z) \ud z \geq \frac{\lambda^p}{p!} \|f\|_{L^p(g \ud z)}^p \geq \frac{\lambda^p}{p^p} \|f\|_{L^p(g \ud z)}^p.
\end{equation*}
Hence, in order to obtain the convergence of the exponential integral, we choose
\begin{equation*}
    \lambda(t)=\frac{1+t}{2CB_1}
\end{equation*}
to conclude that
\begin{align*}
    \sup_{p \geq 1} \frac{\|f\|_{L^p(g \ud z)}}{p} \leq &\, \frac{1}{\lambda} \int_{\mathbb{D}} e^{\lambda f}g(z) \ud z\\
    \leq &\,\frac{2CB_1}{1+t} \int_{\mathbb{D}} \exp \Big(\frac{B_2}{2}-\frac{1}{2}\log(1+t)-\frac{1}{2}\log g(z) \Big) g(z) \ud z \\
    \leq &\,\frac{C}{1+t} \int_{\mathbb{D}} \frac{1}{\sqrt{1+t}} \sqrt{g(z)} \ud z \leq \frac{C}{1+t},
\end{align*}
where we use the Gaussian upper bound of $g$ as in Lemma \ref{gaussianabove} below. Applying the Donsker-Varadhan inequality in Lemma \ref{gibbs} for the test function
\begin{equation*}
    \Phi=\frac{1}{N^2}\sum_{i,j=1}^N \phi(z,w)
\end{equation*}
and for some $\eta(t)=C^{-1}(1+t)$, and using the Large Deviation Theorem \ref{deviation2}, we bound the second term \eqref{thm2} by
\begin{equation*}
    \frac{C}{1+t}\Big(H_N(G^N|g^N)+\frac{1}{N}\Big).
\end{equation*}

\textbf{Control of \eqref{thm1} :} For the first term \eqref{thm1}, we first use a Cauchy-Schwarz argument to bound with the normalized relative Fisher information and some error term:
\begin{align}
    &\frac{\sigma}{4N} \sum_{i=1}^N \int_{\mathbb{D}^N} G^N \Big|\nabla_{x_i} \log \frac{G^N}{g^N}\Big|^2 \ud Z^N \label{fisher}\\
    +&\frac{2}{N\sigma} \sum_{i=1}^N \int_{\mathbb{D}^N} G^N \Big|\frac{1}{N} \sum_{j=1}^N \Big(m_jV(x_i-x_j)-V\ast \omega(x_i)\Big)\Big|^2|\nabla \log g(m_i,x_i)|^2 \ud Z^N.\label{error}
\end{align}
The first term \eqref{fisher} is cancelled by the dissipation term in Lemma \ref{evolution}. Therefore, we only need to control the second term \eqref{error}. Due to symmetry, we rewrite it into
\begin{equation*}
    \int_{\mathbb{D}^N} G^N\Big|\frac{1}{N}\sum_{j=1}^N \Big(m_jV(x_1-x_j)-V\ast \omega(x_1)\Big)\Big|^2|\nabla \log g(m_1,x_1)|^2 \ud Z^N.
\end{equation*}

\textbf{Step 1:} First we extract the terms with $j=1$ in the summation by
\begin{align*}
    \int_{\mathbb{D}^N} G^N \Big|-\frac{1}{N}V\ast \omega(x_1)+\frac{1}{N}\sum_{j=2}^N \Big(m_jV(x_1-x_j)-V\ast \omega(x_1)\Big)\Big|^2\nabla \log g(m_1,x_1)|^2 \ud Z^N.
\end{align*}
Using the simple Cauchy-Schwarz inequality, we bound by
\begin{align*}
    &\frac{1}{N^2}\int_{\mathbb{D}^N} G^N |V\ast \omega(x_1)|^2|\nabla \log g(m_1,x_1)|^2 \ud Z^N\\
    +&\frac{2}{N^2}\int_{\mathbb{D}^N} G^N |V\ast \omega(x_1)|\Big|\sum_{j=2}^N \Big(m_jV(x_1-x_j)-V\ast \omega(x_1)\Big)\Big||\nabla \log g(m_1,x_1)|^2 \ud Z^N\\
    +&\int_{\mathbb{D}^N} G^N \Big|\frac{1}{N}\sum_{j=2}^N \Big(m_jV(x_1-x_j)-V\ast \omega(x_1)\Big)\Big|^2|\nabla \log g(m_1,x_1)|^2 \ud Z^N.
\end{align*}
The first two terms are controlled by crude bounds
\begin{equation*}
    \frac{C}{N}\int_{\mathbb{D}^N} |\nabla \log g(m_1,x_1)|^2 G^N \ud Z^N=\frac{C}{N}\int_{\mathbb{D}} |\nabla \log g(m_1,x_1)|^2 G^{N,1} \ud z_1.
\end{equation*}
Applying the Donsker-Varadhan inequality in Lemma \ref{gibbs} with $N=1$ and $\eta=(1+t)/(2B_1)$, and using the Gaussian upper bound for $g$ as in Lemma \ref{gaussianabove} below, we have
\begin{align*}
    &\frac{C}{N}\int_{\mathbb{D}} |\nabla \log g(m_1,x_1)|^2 G^{N,1} \ud z_1\\
    \leq &\, \frac{C}{N}\Big(\frac{1}{\eta}H_1(G^{N,1}|g)+\frac{1}{\eta}\log \int_{\mathbb{D}} \exp\Big(\frac{\eta B_1}{1+t}(B_2-\log(1+t)-\log g(z))\Big) g(z) \ud z\Big)\\
    \leq &\, \frac{C}{1+t}\Big(\frac{1}{N}H_N(G^N|g^N)+\frac{1}{N}\log \int_{\mathbb{D}} \frac{C}{\sqrt{1+t}} \sqrt{g(z)} \ud z\Big)\\
    \leq &\, \frac{C}{1+t} \Big( H_N(G^N|g^N)+\frac{1}{N} \Big).
\end{align*}

\textbf{Step 2:} Now we focus on the remaining term
\begin{equation*}
    \int_{\mathbb{D}^N} \Big|\frac{1}{N}\sum_{j=2}^N \Big(m_jV(x_1-x_j)-V\ast \omega(x_1)\Big)\Big|^2|\nabla \log g(m_1,x_1)|^2G^N \ud Z^N.
\end{equation*}
We shall again divide the terms in the summation into two parts, one independent of $x_j$, the other one with uniform bound in $x_1$ after multiplied with $\nabla \log g(m_1,x_1)$:
\begin{align*}
    &\, J_1+J_2\\
    \triangleq &\, \int_{\mathbb{D}^N} \Big|\frac{1}{N}\sum_{j=2}^N \Big(m_jV(x_1-x_j)-m_jV(x_1)+\E \M \cdot V(x_1)-V\ast \omega(x_1)\Big)\Big|^2|\nabla \log g(m_1,x_1)|^2G^N \ud Z^N\\
    +&\, \int_{\mathbb{D}^N} \Big|\frac{1}{N}\sum_{j=2}^N \Big(m_jV(x_1)-\E \M \cdot V(x_1)\Big)\Big|^2|\nabla \log g(m_1,x_1)|^2G^N \ud Z^N.
\end{align*}
In the following two steps, we shall deal with $J_1$ and $J_2$ respectively. 

\textbf{Step 3:} We first give an estimate for $J_1$ using our logarithmic gradient estimate Proposition \ref{logarithmic} and the extended version of the Law of Large Numbers Theorem \ref{lln2}. We need to prove that the test function in $J_1$ is uniformly bounded in $z_1$. 

For $|x_1| \leq 2|x_j|$ we have
\begin{equation*}
    |m_jV(x_1-x_j)-m_jV(x_1)+\E \M \cdot V(x_1)-V\ast \omega(x_1)|^2 \leq C
\end{equation*}
by trivial bounds. Consequently,
\begin{align*}
    &|m_jV(x_1-x_j)-m_jV(x_1)+\E \M \cdot V(x_1)-V\ast \omega(x_1)|^2|\nabla \log g(m_1,x_1)|^2\\
    \leq &\,C\sup_{m_1,|x_1| \leq 2|x_j|} |\nabla \log g(m_1,x_1)|^2\\
    \leq &\, \frac{C}{1+t} \Big(1+\frac{|x_j|^2}{1+t} \Big),
\end{align*}
thanks to the logarithmic gradient estimate in Proposition \ref{logarithmic} below. 

For $|x_1| \geq 2|x_j|$ we divide the expression in the large bracket in $J_1$ into two parts. For the first part, we have
\begin{align*}
    |m_jV(x_1-x_j)-m_jV(x_1)| \leq C|x_j|\sup_{s \in [0,1]}|\nabla V(x_1-sv_j)| \leq C\frac{|x_j|}{|x_1|},
\end{align*}
where we use the mean-value inequality and that
\begin{equation*}
    \sup_{s \in [0,1]} |\nabla V(x_1-sy)| \leq \frac{C}{|x_1|} \text{ for any } |x_1|\geq 2|y|.
\end{equation*}
For the second part, we have
\begin{equation*}
    \E \M \cdot V(x_1)-V\ast \omega(x_1)=\int_{\mathbb{D}} m\Big(V(x_1)-V(x_1-y)\Big)g(m,y) \ud m \ud y,
\end{equation*}
which is controlled by
\begin{align*}
    &\,C\int_{|x_1| \leq 2|y|} mg(m,y)\ud m \ud y+C\int_{|x_1| \geq 2|y|} m|y|g(m,y) \sup_{s \in [0,1]}|\nabla V(x_1-sy)|\ud m \ud y\\
    \leq &\,C+C \int_{|x_1| \geq 2|y|} mg(m,y)\frac{|y|}{|x_1|} \ud m \ud y\\
    \leq &\,C\Big( 1+\frac{\sqrt{1+t}}{|x_1|} \Big),
\end{align*}
where we further use the mean value inequality and that
\begin{equation*}
    \sup_{s \in [0,1]} |\nabla V(x_1-sy)| \leq \frac{C}{|x_1|} \text{ for any } |x_1|\geq 2|y|
\end{equation*}
and the Gaussian upper bound of $g(z)$ as in Lemma \ref{gaussianabove} below. This will deduce that for $|x_1| \geq 2|x_j|$:
\begin{equation*}
    |m_jV(x_1-x_j)-m_jV(x_1)+\E \M \cdot V(x_1)-V\ast \omega(x_1)|^2 \leq C \Big(1+\frac{|x_j|^2}{|x_1|^2}+\frac{1+t}{|x_1|^2} \Big).
\end{equation*}
Consequently,
\begin{align*}
    &|m_jV(x_1-x_j)-m_jV(x_1)+\E \M \cdot V(x_1)-V\ast \omega(x_1)|^2|\nabla \log g(m_1,x_1)|^2\\
    \leq &\, \frac{C}{1+t}+C\Big( 1+\frac{|x_j|^2}{|x_1|^2}+\frac{1+t}{|x_1|^2} \Big)\frac{|x_1|^2}{(1+t)^2}\\
    \leq &\, \frac{C}{1+t} \Big(1+\frac{|x_j|^2}{1+t} \Big).
\end{align*}
To this end, we conclude that
\begin{equation*}
    J_1=\frac{1}{N}\int_{\mathbb{D}^N} \Big|\frac{1}{\sqrt{N}}\sum_{j=2}^N \psi(z_1,z_j)\Big|^2 G^N \ud Z^N,
\end{equation*}
where the test function
\begin{equation*}
    \psi(z_1,z_j)=\Big(m_jV(x_1-x_j)-m_jV(x_1)+\E \M \cdot V(x_1)-V\ast \omega(x_1)\Big)\nabla \log g(m_1,x_1)
\end{equation*}
satisfies the growth condition
\begin{equation*}
    \sup_{z_1}|\psi(z_1,z_j)| \leq C\Big(\frac{1}{\sqrt{1+t}}+\frac{|x_j|}{1+t}\Big).
\end{equation*}
It is also direct to check that $\psi$ satisfies the cancelling condition \eqref{cancel3}. Now we conclude with the Donsker-Varadhan inequality Lemma \ref{gibbs} and the extended version of the Law of Large Numbers Theorem \ref{lln2} with some $\eta(t)=C^{-1}\sqrt{1+t}$:
\begin{align*}
    J_1 &\leq \frac{1}{\eta^2} H_N(G^N|g^N)+\frac{1}{\eta^2 N}\log \int_{\mathbb{D}^N} \exp\Big(\frac{\eta^2}{N}\sum_{j,k=1}^N \psi(z_1,z_j)\psi(z_1,z_k)\Big) \ud Z^N\\
    &\leq \frac{C}{1+t}\Big(H_N(G^N|g^N)+\frac{1}{N}\Big).
\end{align*}

\textbf{Step 4:} We now give an estimate for $J_2$ using its special separation structure, namely
\begin{equation*}
    J_2\leq \frac{C}{N^2}\int_{\mathbb{D}^N}\Big|\sum_{j=2}^N (m_j-\E \M)\Big|^2\Big(\frac{1}{1+t}+\frac{|x_1|^2}{(1+t)^2}\Big) G^N \ud Z^N.
\end{equation*}
We shall first integrate with respect to the variable $z_1$, since now the term in the summation is no longer dependent on $x_j$. Denote by $P_{z_2,\cdots,z_N}^{N,1}(z_1)$ the conditional probability measure of $z_1$ under $(z_2,\cdots,z_N)$, i.e.
\begin{equation*}
    P_{z_2,\cdots,z_N}^{N,1}(z_1)=\mathbb{P}(Z_1|Z_2,\cdots,Z_N).
\end{equation*}
By definition, we can simply write that
\begin{equation*}
    P_{z_2,\cdots,z_N}^{N,1}(z_1)=\frac{G^N(z_1,\cdots,z_N)}{G^{N,N-1}(z_2,\cdots,z_N)}.
\end{equation*}
Now we are able to rewrite $J_2$ as
\begin{equation*}
    J_2 \leq \frac{C}{N^2} \int_{\mathbb{D}^{N-1}} \Big|\sum_{j=2}^N (m_j-\E \M)\Big|^2 G^{N,N-1} \ud Z^{N-1} \int_{\mathbb{D}} \Big(\frac{1}{1+t}+\frac{|x_1|^2}{(1+t)^2}\Big) P_{z_2,\cdots,z_N}^{N,1}(z_1) \ud z_1.
\end{equation*}
Using our Donsker-Varadhan inequality Lemma \ref{gibbs} for one variable with some $\eta(t)=C^{-1}(1+t)$, we have
\begin{align*}
    &\int_{\mathbb{D}} \Big(\frac{1}{1+t}+\frac{|x_1|^2}{(1+t)^2}\Big) P_{z_2,\cdots,z_N}^{N,1}(z_1) \ud z_1\\
    \leq &\,\frac{1}{1+t}+\frac{1}{\eta} H_1(P_{z_2,\cdots,z_N}^{N,1}|g)+\frac{1}{\eta} \log \int_{\mathbb{D}} \exp\Big(\frac{\eta|x_1|^2}{(1+t)^2}\Big) g(m_1,x_1) \ud z_1 \\
    \leq &\,\frac{C}{1+t} H_1(P_{z_2,\cdots,z_N}^{N,1}|g)+\frac{C}{1+t}.
\end{align*}
Hence our estimate for $J_2$ writes
\begin{align*}
    J_2 \leq &\,\frac{C}{1+t}\frac{1}{N^2} \int_{\mathbb{D}^{N-1}} \Big|\sum_{j=2}^N (m_j-\E \M)\Big|^2 G^{N,N-1} \ud z_2 \cdots \ud z_N\\
    +&\, \frac{C}{1+t} \int_{\mathbb{D}^{N-1}}  H_1(P_{z_2,\cdots,z_N}^{N,1}|g) G^{N,N-1} \ud z_2 \cdots \ud z_N.
\end{align*}
The first integral is $O(1/N)$ due to the time-independence of $\M^N$ and the fact that $G_0^{N,N-1}=g_0^{\otimes (N-1)}$ (one can also verify by another application of the Donsker-Varadhan inequality Lemma \ref{gibbs} and the Law of Large Numbers Theorem \ref{lln}, since all terms in the summation are bounded; this will also work for the non-factorized initial data case), while the second integral is controlled by $CH_N(G^N|g^N)$ using the chain rule of relative entropy Lemma \ref{chain}, with
\begin{equation*}
    K_x^1=P_{z_2,\cdots,z_N}^{N,1}, \quad K_x^2=g, \quad m_1=G^{N,N-1}, \quad m_2=g^{N-1}.
\end{equation*}
Combining all discussions above, we have successfully reached our goal
\begin{equation*}
    \frac{\ud}{\ud t} H_N(G^N|g^N) \leq \frac{C}{1+t}\Big(H_N(G^N|g^N)+\frac{1}{N}\Big).
\end{equation*}

By the Gr\"onwall inequality, we obtain an upper bound for the global relative entropy $H_N(G^N|g^N)$ from the above differential inequality:
\begin{equation*}
    H_N(G^N|g^N) \leq (1+t)^C \Big( H_N(G_0^N|g_0^N)+\frac{1}{N} \Big).
\end{equation*}
This proves Theorem \ref{thm1}.  By the sub-additivity property of the relative entropy and the condition on the initial data (see Remark \ref{initialdata}), we have
\begin{equation*}
    H_k(G^{N,k}|g^k) \leq H_N(G^N|g^N) \leq \frac{1}{N}(1+t)^C.
\end{equation*}
Using the CKP inequality, we have the control of the $L^1$ distance:
\begin{equation*}
    \|\omega^{N,k}-\omega^{\otimes k}\|_{L^1(\R^{2k})} \leq \|G^{N,k}-g^k\|_{L^1(\mathbb{D}^{k})} \leq \sqrt{2kH_k(G^{N,k}|g^k)} \leq \frac{\sqrt{k}}{\sqrt{N}} M(1+t)^M,
\end{equation*}
which completes the proof of Corollary \ref{global}.

\begin{rmk}[Sketch of Proof of the Torus Case]
    To close this subsection, we make a short sketch of the proof of the global relative entropy control on the torus case, without the independence assumption on $\M$ and $X_0$. It is easy to observe that the time evolution of the relative entropy is invariant in this case, once we change the integration domain into the torus. Hence we only need to bound the two terms \eqref{thm2} and \eqref{error}.

    If we assume on the initial limit density $g_0$ that it has some positive lower bound
    \begin{equation*}
        g_0(m,x) \geq \lambda>0,
    \end{equation*}
    which is the same as \cite{shao2024quantitative}, then it is direct to show that $\nabla \log g$ and $\nabla^2 \log g$, and thus the test functions in \eqref{thm2} and \eqref{error}, are uniformly bounded from above by some universal constant. Hence this is essentially the same condition with \cite{shao2024quantitative}. Moreover, by applying the time-uniform LSI on the limit density, as pointed out in \cite{guillin2024uniform}, one can even obtain uniform-in-time propagation of chaos results for this torus model.
\end{rmk}

\subsection{Local Relative Entropy Control}\label{control2}

In this subsection we prove our second main estimate on the right-hand side of Lemma \ref{evolutionlocal}, say the time derivative of the local relative entropy. Thanks again to the non-positive dissipation term, we only need to control the rest terms, which we recall
\begin{align*}
    &J_1+J_2+J_3\\
    \triangleq&-\frac{1}{k} \sum_{i=1}^k \frac{1}{N} \sum_{j=1}^k \int_{\mathbb{D}^k} G^{N,k} \Big(m_j K(x_i-x_j)-K \ast \omega (x_i)\Big) \cdot \nabla \log g(m_i,x_i) \ud Z^k\\
    &-\frac{1}{k}\sum_{i=1}^k \frac{N-k}{N} \int_{\mathbb{D}^{k+1}} G^{N,k+1} \Big(m_{k+1}K(x_i-x_{k+1})-K \ast \omega(x_i) \Big) \cdot \nabla \log g(m_i,x_i) \ud Z^{k+1}\\
    &+\frac{1}{k}\sum_{i=1}^k \frac{N-k}{N} \int_{\mathbb{D}^{k+1}} G^{N,k+1} \Big(m_{k+1}K(x_i-x_{k+1}) \cdot \nabla_{x_i} \log G^{N,k}\Big) \ud Z^{k+1}.
\end{align*}

\textbf{Step 1: Estimate of $J_1$.}
The first self-contained term $J_1$ is similar to the quantity that we estimate in the previous subsection. Using again the integration by parts with $K=\nabla \cdot V$ where
\begin{equation*}
    V(x)=\Big(-\frac{1}{2\pi}\arctan\frac{x_1}{x_2}\Big)\text{Id}
\end{equation*}
is bounded, we eliminate the singularity and obtain
\begin{align*}
    &\frac{1}{kN} \sum_{i,j=1}^k \int_{\mathbb{D}^k} \Big(m_jV(x_i-x_j)-V \ast \omega(x_i)\Big):\Big(\nabla_{x_i}g^k \otimes \nabla_{x_i}\frac{G^{N,k}}{g^k}\Big) \ud Z^k\\
    +&\frac{1}{kN} \sum_{i,j=1}^k \int_{\mathbb{D}^k} G^{N,k}\Big(m_jV(x_i-x_j)-V\ast \omega(x_i)\Big):\frac{\nabla_{x_i}^2 g^k}{g^k} \ud Z^k.
\end{align*}
It is easy to observe that once we extract the factor $k/N$ from the above expression, the remaining quantity has exactly the same structure as the time derivative of the global relative entropy appearing in the previous subsection, see \eqref{thm1} and \eqref{thm2}. Hence we shall simply copy the same estimate, but only make a small scaling of $k/N$ on every $\eta (t)$ appearing before, to control $J_1$ by
\begin{align*}
    &\frac{\sigma}{4}I_k(G^{N,k}|g^k)+\frac{k}{N} \cdot \frac{C}{1+t}\frac{N}{k}\Big(H_k(G^{N,k}|g^k)+\frac{1}{k}\frac{k^2}{N^2}\Big)\\
    =&\,\frac{\sigma}{4}I_k(G^{N,k}|g^k)+\frac{C}{1+t}\Big(H_k(G^{N,k}|g^k)+\frac{k}{N^2}\Big),
\end{align*}
where in the second term of the first line, the factor $N/k$ is due to the scaling of $\eta(t)$ and the factor $k^2/N^2$ is due to the extended Large Deviation Theorem \ref{deviation2}, where the upper bound $\gamma$ is rescaled corresponding to $\eta(t)$.

\textbf{Step 2: Estimate of $J_2+J_3$.} As for the rest term $J_2+J_3$, which is no longer self-contained since the integral includes the $(k+1)$-marginal function, we adapt the idea of \cite{wang2025sharp}. First, we use the symmetry property and regroup the terms into
\begin{align*}
    \frac{N-k}{N}\int_{\mathbb{D}^{k+1}} G^{N,k}\Big(m_{k+1}K(x_1-x_{k+1})\frac{G^{N,k+1}}{G^{N,k}}-K \ast \omega(x_1)\Big) \cdot \nabla_{x_1} \log \frac{G^{N,k}}{g^k} \ud Z^{k+1},
\end{align*}
where we have also used the divergence-free property of the kernel $K$. Now we first perform a Cauchy-Schwarz argument to control by
\begin{align*}
    &\frac{\sigma}{4}\int_{\mathbb{D}^k} G^{N,k}\Big|\nabla_{x_1}\log \frac{G^{N,k}}{g^k}\Big|^2 \ud Z^k\\
    +&\frac{1}{\sigma}\int_{\mathbb{D}^{k+1}} G^{N,k} \Big|m_{k+1}K(x_1-x_{k+1})\frac{G^{N,k+1}}{G^{N,k}}-K \ast \omega(x_1)\Big|^2 \ud Z^{k+1}.
\end{align*}
The first part is exactly the relative Fisher information and will be cancelled again by the non-positive dissipation term on the right-hand side of \eqref{evolutionlocal}. Therefore, we only need to focus on the second part, which we shall rewrite in
\begin{equation}\label{mainterm}
    \frac{1}{\sigma}\int_{\mathbb{D}^k} G^{N,k} \Big|\int_{\mathbb{D}} m_{k+1}K(x_1-x_{k+1})\Big(\frac{G^{N,k+1}}{G^{N,k}}-g(m_{k+1},x_{k+1})\Big) \ud z_{k+1} \Big|^2 \ud Z^k.
\end{equation}
The inner integral is in the form of the difference between two probability density functions testing on some singular function, which we cannot directly bound by some total variation norm using the CKP inequality. However, we can perform integration by parts to eliminate the singularity and make use of the relative Fisher information to deal with the extra derivatives. This has been observed by \cite{wang2025sharp} under the condition that $\nabla \log g \in L^\infty$, which makes sense in the torus case. Now we generalize this idea to our setting, where the logarithmic gradient growth estimate holds as in Proposition \ref{logarithmic} below, by introducing the weighted CKP inequality initiated by Bolley-Villani \cite{Bolley2005}. We are coming up with the following transport-type estimate.
\begin{lemma}
    For all $K=\nabla \cdot V$ with $V \in L^\infty(\R^d;\R^d \times \R^d)$ and all regular enough probability density functions $m_1,m_2$ on $\R^d$, we have for any $\lambda>0$,
    \begin{align*}
        |\langle K, m_1-m_2 \rangle| &\, \leq \| V\|_{L^\infty} \sqrt{I(m_1|m_2)}\\
        &\, +\lambda^{-1}\sqrt{1+\log \int e^{\lambda^2\|V\|_{L^\infty}^2|\nabla \log m_2|^2} \ud m_2}\sqrt{2H(m_1|m_2)}.
    \end{align*}
\end{lemma}
\begin{proof}
    This result is a simple combination of the integration by parts process of \cite[Proposition 7]{wang2025sharp} and the weighted CKP inequality, see for example in \cite{Bolley2005} or \cite[Section 6]{lacker2021hierarchies}.
    \begin{align*}
        |\langle K, m_1-m_2 \rangle|=& \,|\langle V, \nabla m_1-\nabla m_2 \rangle|\\
        \leq & \, \int |V|\Big|\frac{\nabla m_1}{m_1}-\frac{\nabla m_2}{m_2}\Big| \ud m_1+\Big|\int V\frac{\nabla m_2}{m_2} \ud(m_1-m_2)\Big|\\
        \leq & \, \|V\|_{L^\infty} \sqrt{\int \Big|\nabla \log \frac{m_1}{m_2}\Big|^2 \ud m_1}+\lambda^{-1} |\langle m_1-m_2, \lambda V\nabla \log m_2 \rangle|\\
        \leq & \, \| V \|_{L^\infty} \sqrt{I(m_1|m_2)}+\lambda^{-1}\sqrt{2\Big(1+\log \int e^{\lambda^2|V|^2|\nabla \log m_2|^2} \ud m_2\Big)H(m_1|m_2)}.
    \end{align*}
\end{proof}
Now we apply the above lemma to our main term \eqref{mainterm}, with
\begin{equation*}
    m_1=\frac{G^{N,k+1}}{G^{N,k}}, \quad m_2=g
\end{equation*}
to obtain that
\begin{align*}
    & \, \frac{1}{\sigma}\int_{\mathbb{D}^k} G^{N,k} \Big|\int_{\mathbb{D}} m_{k+1}K(x_1-x_{k+1})\Big(\frac{G^{N,k+1}}{G^{N,k}}-g(m_{k+1},x_{k+1})\Big) \ud z_{k+1} \Big|^2 \ud Z^k\\
    \leq & \, \frac{1}{\sigma}\int_{\mathbb{D}^k} G^{N,k} \Big[A^2\|V\|_{L^\infty}^2 I\Big(\frac{G^{N,k+1}}{G^{N,k}}\Big|g\Big)\\
    &\, +2\lambda^{-2}\Big(1+\log \int_{\mathbb{D}} e^{\lambda^2\|V\|_{L^\infty}^2|\nabla \log g|^2} g(z_{k+1}) \ud z_{k+1}\Big) \cdot H\Big(\frac{G^{N,k+1}}{G^{N,k}}\Big|g\Big)\Big] \ud Z^k.
\end{align*}
We recall the classical chain rule or the towering property of the relative entropy and the relative Fisher information:
\begin{align*}
    \int_{\mathbb{D}^k} G^{N,k} I\Big(\frac{G^{N,k+1}}{G^{N,k}}\Big|g\Big) \ud Z^k&=I_{k+1}(G^{N,k+1}|g^{k+1}),\\
    \int_{\mathbb{D}^k} G^{N,k} H\Big(\frac{G^{N,k+1}}{G^{N,k}}\Big|g\Big) \ud Z^k&=(k+1)H_{k+1}(G^{N,k+1}|g^{k+1})-kH_k(G^{N,k}|g^k),
\end{align*}
where we emphasize that the entropy and Fisher information are normalized quantities with the index. As we have the logarithmic gradient estimate Proposition \ref{logarithmic} and the Gaussian upper bound Lemma \ref{gaussianabove}, we may choose $\lambda=\sqrt{1+t}/C$ to ensure the convergence of the integral inside the logarithm, where $C$ depends only linearly on $\|V\|_{L^\infty}$ and the initial bounds. Hence we conclude by
\begin{align*}
    J_2+J_3 & \, \leq \frac{\sigma}{4}I_k(G^{N,k}|g^k)+\frac{A^2\|V\|_{L^\infty}^2}{\sigma}I_{k+1}(G^{N,k+1}|g^{k+1})\\
    & \, +\frac{C}{1+t}\Big((k+1)H_{k+1}(G^{N,k+1}|g^{k+1})-kH_k(G^{N,k}|g^k)\Big).
\end{align*}

\textbf{Step 3: Solving the ODE Hierarchy.} Combining the two steps above, we conclude that the local relative entropy of all orders form some kind of ODE hierarchy, in the sense that the time derivative of the $k$-order entropy is controlled by the quantity containing the $(k+1)$-order entropy. Now we only need to solve this hierarchy system to conclude our desired estimate. For simplicity, denote that
\begin{equation*}
    I_k \triangleq I_k(G^{N,k}|g^k), \quad H_k \triangleq H_k(G^{N,k}|g^k).
\end{equation*}
Then we have the following estimate:
\begin{equation}
\begin{aligned}
    \frac{\ud}{\ud t}H_k \leq &\, -\frac{\sigma}{2}I_k+\frac{A^2\|V\|_{L^\infty}^2}{\sigma}I_{k+1}1_{k<N}+\frac{C}{1+t}H_k\\
    &\, +\frac{C}{1+t}((k+1)H_{k+1}-kH_k)1_{k<N}+\frac{C}{1+t}\frac{k}{N^2}.
\end{aligned}
\end{equation}
Inspired by \cite{wang2025sharp}, we introduce the following new ODE hierarchy proposition to deal with the evolution of the local relative entropy.
\begin{prop}[New ODE Hierarchy]\label{odehierarchy}
    Let $T>0$ and let $x_k, y_k \in C^1([0,T];[0,\infty))$ for any $1 \leq k \leq N$. Suppose that $x_{k+1} \geq x_k$ for all $1 \leq k \leq N-1$. Assume that there exists real numbers $c_1>c_2\geq 0$ and some non-negative time decay function $h(t)$ such that for all $t \in [0,T]$,
    \begin{equation*}
        \begin{cases}
            &\frac{\ud}{\ud t}x_k \leq -c_1y_k+c_2y_{k+1}1_{k<N}+h(t)x_k+kh(t)(x_{k+1}-x_k)1_{k<N}+\frac{k^2}{N^2}h(t),\\
            &x_k(0) \leq C\frac{k^2}{N^2}.
        \end{cases}
    \end{equation*}
    Suppose that we also have the sub-optimal \textit{a priori} estimates
    \begin{equation*}
        x_k(t) \leq Ce^{\varphi(t)}\frac{k}{N},
    \end{equation*}
    then there exists some $M>0$ such that for all $t \in [0,T]$ we have
    \begin{equation*}
        x_k(t) \leq Me^{5\varphi(t)}\frac{k^2}{N^2},
    \end{equation*}
    where the growth-in-time function is given by $\varphi(t)=C\int_0^t h(s) \ud s$.
\end{prop}
The proof of Proposition \ref{odehierarchy} is rather technical and will be presented in the next subsection. We keep on finishing the main proof admitting the above proposition. Under the assumption for the viscosity constant $\sigma$ that $\sigma>\sqrt{2}A\|V\|_{L^\infty}$, we can apply this ODE hierarchy result for
\begin{equation*}
    x_k(t)=kH_k(t), \quad y_k(t)=kI_k(t), \quad c_1=\frac{\sigma}{2}, \quad c_2=\frac{A^2\|V\|_{L^\infty}^2}{\sigma}, \quad h(t)=\frac{C}{1+t},
\end{equation*}
to arrive at our desired main result, say
\begin{equation*}
    H_k(t) \leq M(1+t)^M\frac{k}{N^2}.
\end{equation*}

Using the CKP inequality, we have the control of the $L^1$ distance:
\begin{equation*}
    \|\omega^{N,k}-\omega^{\otimes k}\|_{L^1(\R^{2k})} \leq \|G^{N,k}-g^k\|_{L^1(\mathbb{D}^{k})} \leq \sqrt{2kH_k(G^{N,k}|g^k)} \leq \frac{k}{N} M(1+t)^M,
\end{equation*}
which completes the proof of Corollary \ref{local}.

\subsection{Solving the New ODE Hierarchy}

The ODE hierarchy Proposition \ref{odehierarchy} is somehow an extension of the similar result in \cite[Proposition 5]{wang2025sharp}, where the coefficients appearing in the ODE are either constants or exponential functions in time $t$, hence allowing S. Wang to apply the iterated integral estimates in \cite{lacker2021hierarchies}. However, in our cases, the growth in time function is more complicated from the logarithmic estimates, which forces us to deal with more general iterated integrals. 

\textbf{Step 1: Change of Variables.} We take the index
\begin{equation*}
    i_0 \triangleq \max\Big(1, \inf\Big\lbrace i>0: \frac{i^5}{(i+1)^5} \geq \frac{c_2}{c_1} \Big\rbrace \Big),
\end{equation*}
which exists since $\lim_{i \rightarrow \infty} i^5/(i+1)^5=1>c_2/c_1$. We induce the following new variable 
$z_k(t)$ as a weighted mixture of $x_i(t)$ for all $i \geq k$:
\begin{equation*}
    z_k(t) \triangleq \sum_{i=k}^N \frac{x_i(t)}{(i-k+i_0)^5}.
\end{equation*}
Repeating the estimates of \cite{wang2025sharp}, we arrive at that for any $k \leq N/2$,
\begin{equation*}
    \frac{\ud}{\ud t}z_k \leq f(t)z_k+kf(t)(z_{k+1}-z_k)+\frac{k^2}{N^2}f(t),
\end{equation*}
where $f(t)=Ch(t)$ for some universal constant $C$. Also, the initial data satisfy for any $k \leq N/2$,
\begin{equation*}
    z_k(0) \leq C\frac{k^2}{N^2}.
\end{equation*}
In order to cancel the first term $f(t)z_k$ on the right-hand side, we introduce the growth-in-time function
\begin{equation*}
    \varphi(t)=\int_0^t f(s) \ud s
\end{equation*}
and define the new variable by
\begin{equation*}
    w_k(t)=e^{-\varphi(t)}z_k(t).
\end{equation*}
Then we have
\begin{equation*}
    \frac{\ud}{\ud t}w_k \leq kf(t)(w_{k+1}-w_k)+\frac{k^2}{N^2}f(t).
\end{equation*}
Our initial condition now becomes for any $k \leq N/2$,
\begin{equation*}
    w_k(0) \leq C\frac{k^2}{N^2},
\end{equation*}
while the sub-optimal \textit{a priori} estimate reads
\begin{equation*}
    w_k(t) \leq C\varphi(t)e^{-\varphi(t)}\frac{k}{N}.
\end{equation*}

\textbf{Step 2: Iterating the Gr\"onwall's Inequality.} Now we are able to first apply the Gr\"onwall's inequlity to obtain the control of $w_k(t)$ by
\begin{equation*}
    w_k(t) \leq e^{-k\varphi(t)}w_k(0)+\int_0^t e^{-k(\varphi(t)-\varphi(s))}f(s)\Big(\frac{k^2}{N^2}+kw_{k+1}(s)\Big) \ud s.
\end{equation*}
Repeating this procedure for $w_{k+1}(s)$, we derive the next order hierarchy as
\begin{align*}
    w_k(t) &\leq e^{-k\varphi(t)}w_k(0)+\frac{k^2}{N^2}\int_0^t e^{-k(\varphi(t)-\varphi(s))}f(s) \ud s\\
    &+k\int_0^t e^{-k(\varphi(t)-\varphi(s))}f(s) e^{-(k+1)\varphi(s)}w_{k+1}(0) \ud s\\
    &+k\frac{(k+1)^2}{N^2}\int_0^t e^{-k(\varphi(t)-\varphi(s))} f(s) \int_0^s e^{-(k+1)(\varphi(s)-\varphi(r))}f(r) \ud r \ud s\\
    &+k(k+1)\int_0^t e^{-k(\varphi(t)-\varphi(s))} f(s) \int_0^s e^{-(k+1)(\varphi(s)-\varphi(r))}f(r)w_{k+2}(r) \ud r \ud s.
\end{align*}
By simple iteration of the above formula, we are now arriving at the estimate of $w_k(t)$ with the iterated integrals, i.e.
\begin{align*}
    w_k(t) &\leq \sum_{l=k}^{[N/2]-1} B_k^l(t) w_l(0)+\sum_{l=k}^{[N/2]-1} \frac{l}{N^2}A_k^l(t)+A_k^{[N/2]}(t) \sup_{0 \leq s \leq t}w_{[N/2]}(s)\\
    &\leq C\sum_{l=k}^{[N/2]-1}B_k^l(t)\frac{l^2}{N^2}+\sum_{l=k}^{[N/2]-1}\frac{l}{N^2}A_k^l(t)+C\varphi(t)e^{-\varphi(t)}A_k^{[N/2]}(t),
\end{align*}
where the coefficient terms are given by the iterated integrals for any $k \leq l$,
\begin{equation}\label{akl}
\begin{aligned}
    A_k^l(t_k)=\Big(k(k+1)\cdots l\Big)\int_0^{t_k}\cdots\int_0^{t_l} &\exp\Big(-\sum_{j=k}^l j\big(\varphi(t_j)-\varphi(t_{j+1})\big)\Big)\\ &f(t_{k+1})\cdots f(t_{l+1}) \ud t_{l+1}\cdots\ud t_{k+1},
\end{aligned}
\end{equation}
\begin{equation}\label{bkl}
    \begin{aligned}        
    B_k^l(t_k)=\Big(k(k+1)\cdots (l-1)\Big)\int_0^{t_k}\cdots\int_0^{t_{l-1}} &\exp\Big(-\sum_{j=k}^{l-1} j\big(\varphi(t_j)-\varphi(t_{j+1})\big)-l\varphi(t_l)\Big)\\ &f(t_{k+1})\cdots f(t_{l}) \ud t_{l}\cdots \ud t_{k+1},
    \end{aligned}
\end{equation}
together with $A_k^k(t)=1-e^{-k\varphi(t)}$ and $B_k^k(t)=e^{-k\varphi(t)}$. Also, we have the initial values
\begin{equation*}
    A_k^l(0)=0, \quad B_k^l(0)=
    \begin{cases}
        &1 \quad k=l,\\
        &0 \quad k<l.
    \end{cases}
\end{equation*}

\textbf{Step 3: Estimate of Iterated Integrals.} We highlight that these two iterated integrals above extend the similar expressions studied in \cite{lacker2021hierarchies,wang2025sharp}, where the growth-in-time function is in the specific structure
\begin{equation*}
    \varphi(t)=\gamma t, \quad f(t)=\gamma.
\end{equation*}
For those special cases, Lacker is able to change the formula into some convolution type expression and use the tail estimate of Beta distribution to give the desired estimate. However, in our cases, since the summation term in the exponential is no longer linear, we cannot expect the same convolution translation, but only wish to control it directly. We summarize the results into a single lemma for convenience.
\begin{lemma}[Computations of Iterated Integrals]\label{iteratedintegrals}
    We give the explicit expressions of the iterated integrals \eqref{akl} \eqref{bkl} by
    \begin{align*}
        &A_k^l(t)=\frac{l!}{(l-k)!(k-1)!}\int_{e^{-\varphi(t)}}^1 z^{k-1}(1-z)^{l-k} \ud z,\\
        &B_k^l(t)=C_{l-1}^{k-1}(1-e^{-\varphi(t)})^{l-k}e^{-k\varphi(t)}.
    \end{align*}
\end{lemma}
For the sake of completeness, we shall give the technical proof of Lemma \ref{iteratedintegrals} first and then apply it to complete the main proof of Proposition \ref{odehierarchy}. Those readers who are not so interested in this proof can skip it for the first time of reading.
\begin{proof}[Proof of Lemma \ref{iteratedintegrals}]
    Notice that for any $k<l$, using the expression \eqref{bkl} we have
\begin{align*}
    &\frac{\ud}{\ud t_k}B_k^l(t_k)\\
    =&\, \Big(\prod_{j=k}^{l-1} j \Big)\int_0^{t_k}\cdots\int_0^{t_{l-1}} \exp\Big(-\sum_{j=k}^{l-1} j\big(\varphi(t_j)-\varphi(t_{j+1})\big)-l\varphi(t_l)\Big)\Big(-kf(t_k)f(t_{k+1})\cdots f(t_{l})\Big)\\
    &\, +\Big(\prod_{j=k}^{l-1} j \Big)\int_0^{t_{k+1}}\cdots\int_0^{t_{l-1}} \exp\Big(-\sum_{j=k}^{l-1} j\big(\varphi(t_j)-\varphi(t_{j+1})\big)-l\varphi(t_l)\Big)f(t_{k+1}) \cdots f(t_{l})\Big|_{t_{k+1}=t_k}\\
    =&\, -kf(t_k)B_k^l(t_k)\\
    &\, +\Big(\prod_{j=k}^{l-1} j \Big)\int_0^{t_{k+1}}\cdots \int_0^{t_{l-1}} \exp\Big(-\sum_{j=k+1}^{l-1} j\big(\varphi(t_j)-\varphi(t_{j+1})\big)-l\varphi(t_{l})\Big)f(t_{k+1})\cdots f(t_{l})\Big|_{t_{k+1}=t_k}\\
    =&\, -kf(t_k)B_k^l(t_k)+kf(t_k)B_{k+1}^l(t_k).
\end{align*}
This has given us a key idea that the iterated integral itself solves another much simpler ODE hierarchy, which enables us this time to use the Gr\"onwall's inequality and direct induction to derive the explicit expressions of the complicated integrals. In order to cancel the first term on the right-hand side, we denote that
\begin{equation*}
    R_k^l(t)=e^{k\varphi(t)}B_k^l(t),
\end{equation*}
which solves by simple computations that
\begin{equation*}
    \frac{\ud}{\ud t}R_k^l=ke^{-\varphi(t)}f(t)R_{k+1}^l
\end{equation*}
for any $k<l$. The starting point of the new hierarchy reads that
\begin{equation*}
    R_k^k(t)=e^{k\varphi(t)}B_k^k(t)=1.
\end{equation*}
We prove by induction on $l-k$ that
\begin{equation}\label{rkl}
    R_k^l(t)=C_{l-1}^{k-1}(1-e^{-\varphi(t)})^{l-k}.
\end{equation}
The starting case $l-k=0$ has been verified above. Now assume that \eqref{rkl} holds for any $0 \leq l-k \leq m$ for some $m \geq 0$, then for $l-k=m+1$ we have
\begin{align*}
    R_k^l(t)=&\, R_k^l(0)+\int_0^t ke^{-\varphi(s)}f(s)R_{k+1}^l(s) \ud s\\
    =& \, kC_{l-1}^k \int_0^t (1-e^{-\varphi(s)})^{l-k-1} e^{-\varphi(s)}f(s) \ud s\\
    =&\, \frac{k}{l-k}C_{l-1}^{k}(1-e^{-\varphi(s)})^{l-k}\Big|_0^t\\
    =&\, C_{l-1}^{k-1}(1-e^{-\varphi(t)})^{l-k}.
\end{align*}
This verifies the case for $l-k=m+1$. Hence \eqref{rkl} holds for any $k \leq l$ and thus we conclude that
\begin{equation}
    B_k^l(t)=C_{l-1}^{k-1}(1-e^{-\varphi(t)})^{l-k}e^{-k\varphi(t)}.
\end{equation}
Now we return to the expression of $A_k^l(t)$ in \eqref{akl}, which reads
\begin{align*}
    A_k^l(t_k)=\Big(k(k+1)\cdots l\Big)\int_0^{t_k}\cdots \int_0^{t_l} &\, \exp\Big(-\sum_{j=k}^l j\big(\varphi(t_j)-\varphi(t_{j+1})\big)\Big)\\ &\, f(t_{k+1})\cdots f(t_{l+1}) \ud t_{l+1}\cdots \ud t_{k+1}.
\end{align*}
Integrating with respect to $t_{l+1}$, we obtain that
\begin{align*}
    A_k^l(t_k)=&\, \Big(k(k+1)\cdots (l-1)\Big)\int_0^{t_k}\cdots \int_0^{t_{l-1}} \exp\Big(-\sum_{j=k}^{l-1} j\big(\varphi(t_j)-\varphi(t_{j+1})\big)-l\varphi(t_l)\Big)\\
    &\, \qquad \qquad \qquad \qquad \qquad \qquad \Big(\exp(l\varphi(t_l))-1\Big)f(t_{k+1})\cdots f(t_l)\ud t_{l}\cdots \ud t_{k+1}\\
    =&\, A_k^{l-1}(t_k)-B_k^l(t_k).
\end{align*}
Hence for any $l>k$ we summarize that
\begin{align*}
    A_k^l(t)=&\, A_k^k(t)-\sum_{j=k+1}^l B_k^j(t)\\
    =&\, 1-e^{-k\varphi(t)}-\sum_{j=k+1}^l C_{j-1}^{k-1}(1-e^{-\varphi(t)})^{j-k}e^{-k\varphi(t)}\\
    =&\, 1-e^{-k\varphi(t)}\sum_{j=0}^{l-k} C_{j+k-1}^{k-1} (1-e^{-\varphi(t)})^{j}.
\end{align*}
Using the simple combinatorial identity
\begin{equation*}
    \sum_{j=0}^\infty C_{j+k-1}^{k-1}x^j=\frac{1}{(1-x)^k}
\end{equation*}
for any $x \in (0,1)$, we can rewrite the expression for $A_k^l(t)$ into some form of extra terms in the Taylor's expansion:
\begin{align*}
    A_k^l(t)&\, =e^{-k\varphi(t)}\sum_{j=l-k+1}^\infty C_{j+k-1}^{k-1}x^j\Big|_{x=1-e^{-\varphi(t)}}\\
    &\, =e^{-k\varphi(t)}\frac{(-1)^{l-k}}{(l-k)!}\Big(k(k+1)\cdots l\Big)\int_0^x (y-x)^{l-k}(1-y)^{-l-1} \ud y\Big|_{x=1-e^{-\varphi(t)}}.
\end{align*}
By change of variables $z=\frac{1-x}{1-y}$, one can deduce that
\begin{align*}
    A_k^l(t)=&\,e^{-k\varphi(t)} \frac{(-1)^{l-k}}{(l-k)!}(1-x)^{-k-1}\int_{1-x}^1 z^{k-1}(z-1)^{l-k} \ud z\Big|_{x=1-e^{-\varphi(t)}}\\
    =&\, \frac{l!}{(l-k)!(k-1)!}\int_{e^{-\varphi(t)}}^1 z^{k-1}(1-z)^{l-k} \ud z.
\end{align*}
This concludes the proof of Lemma \ref{iteratedintegrals}.
\end{proof}

\textbf{Step 4: Conclusion of the Proof.} Now we return to the last part of the proof of Proposition \ref{odehierarchy}. Recall that the target term writes
\begin{align*}
    w_k(t) &\leq C\sum_{l=k}^{[N/2]-1}B_k^l(t)\frac{l^2}{N^2}+\sum_{l=k}^{[N/2]-1}\frac{l}{N^2}A_k^l(t)+C\varphi(t)e^{-\varphi(t)}A_k^{[N/2]}(t)\\
    &\leq \frac{C}{N^2}\sum_{l=k}^\infty l^2B_k^l(t)+\frac{1}{N^2}\sum_{l=k}^\infty lA_k^l(t)+C\varphi(t)e^{-\varphi(t)}A_k^{[N/2]}(t).
\end{align*}
We first deal with the second term. Notice that by the Abel transform, we have
\begin{align*}
    \sum_{l=k}^\infty lA_k^l(t) \leq&\, \frac{1}{2} \sum_{l=k}^\infty \Big((l+1)^2-l^2\Big)A_k^l(t)\\
    \leq &\, \frac{1}{2}\sum_{l=k+1}^\infty l^2\Big(A_k^{l-1}(t)-A_k^l(t)\Big)\\
    \leq&\, \frac{1}{2} \sum_{l=k+1}^\infty l^2B_k^l(t).
\end{align*}
Hence it can be combined with the first term. Recall again that
\begin{equation*}
    B_k^l(t)=C_{l-1}^{k-1}(1-e^{-\varphi(t)})^{l-k}e^{-k\varphi(t)}.
\end{equation*}
Hence we compute that
\begin{align*}
    \sum_{l=k}^\infty l^2B_k^l(t)&\, =e^{-k\varphi(t)}\sum_{l=k}^\infty \Big(l(l+1)-l\Big)C_{l-1}^{k-1}(1-e^{-\varphi(t)})^{l-k}\\
    &\, =e^{-k\varphi(t)}\Big(k(k+1)\sum_{l=k}^\infty C_{l+1}^{k+1}(1-e^{-\varphi(t)})^{l-k}-k\sum_{l=k}^\infty C_l^k (1-e^{-\varphi(t)})^{l-k}\Big)\\
    &\, =e^{-k\varphi(t)}\Big(k(k+1)e^{(k+2)\varphi(t)}-ke^{(k+1)\varphi(t)}\Big)\\
    &\, \leq k(k+1)e^{2\varphi(t)}.
\end{align*}
This estimate will handle the first two terms together. As for the last non-summation term, we notice by Lemma \ref{iteratedintegrals} that
\begin{align*}
    A_k^l(t)
    =&\, \frac{l!}{(l-k)!(k-1)!}\int_{e^{-\varphi(t)}}^1 z^{k-1}(1-z)^{l-k} \ud z\\
    =&\, \mathbb{P}(Y>e^{-\varphi(t)}),
\end{align*}
where $Y$ is some random variable with the Beta distribution $B(k,l-k+1)$. By the sub-Gaussian tail estimate of the Beta distribution in \cite[Theorem 2.1]{marchal2017sub}, which reads
\begin{equation*}
    \mathbb{P}(Y>t) \leq \exp\Big(-2(l+2)\Big(t-\frac{k}{l+1}\Big)_+^2\Big).
\end{equation*}
We choose $l=[N/2]$ and consider the following two different cases respectively: if $k$ is large in the sense of $e^{-\varphi(t)} \leq \frac{2k}{l+1}$, then we have
\begin{equation*}
    A_k^l(t) \leq 1 \leq e^{2\varphi(t)}\frac{4k^2}{(l+1)^2} \leq 16e^{2\varphi(t)}\frac{k^2}{N^2};
\end{equation*}
if $k$ is small in the sense of $e^{-\varphi(t)} \geq \frac{2k}{l+1}$, then we have
\begin{equation*}
    A_k^l(t) \leq \exp\Big(-\frac{l+2}{2}e^{-2\varphi(t)}\Big) \leq 32e^{4\varphi(t)}\frac{1}{N^2} \leq 32e^{4\varphi(t)}\frac{k^2}{N^2}.
\end{equation*}
By a little bookkeeping of all the discussions above, we conclude that
\begin{equation*}
    w_k(t) \leq Ce^{4\varphi(t)}\frac{k^2}{N^2}.
\end{equation*}
Tracing back to the original variable $x_k(t)$ gives that
\begin{equation*}
    x_k(t) \leq i_0^5z_k(t) \leq Ce^{\varphi(t)}w_k(t) \leq Ce^{5\varphi(t)}\frac{k^2}{N^2}.
\end{equation*}
This finishes our proof of Proposition \ref{odehierarchy}.

\section{Regularity Estimates}\label{regularity}

After presenting the proof of our main results, we are now turning to providing the key logarithmic growth estimates of the limit density, which are essentially applied in the previous discussions. In this section, we shall first collect some basic regularity estimates of the limit density $g$ as ingredients of our logarithmic estimates. 

First we borrow some long time asymptotic decay regularity estimates of the vorticity function $\omega$ from the classical work of Kato \cite{kato1994navier}, since $\omega$ solves the vorticity formulation of the 2D Navier-Stokes equation with an $L^1$ initial data, as pointed out before. This will give us enough knowledge about the coefficient of the drift term for the limit equation \eqref{fp}.

\begin{lemma}[Asymptotic Decay]\label{omega}
    Assume that $g(t,m,x)$ solves the nonlinear Fokker-Planck limit equation \eqref{fp} with initial data $g_0 \in W_x^{2,1} \cap W_x^{2,\infty}(\mathbb{D})$, then the vorticity function $\omega(t,x)$, which is given by the expectation of $g(t,m,x)$ with respect to the circulation variable $m$:
    \begin{equation*}
        \omega(t,x)=\int_\R mg(t,m,x) \ud m,
    \end{equation*}
    satisfies the asymptotic decay estimates:
    \begin{equation*}
        \| \nabla^k K \ast \omega(t,\cdot) \|_{L^\infty} \leq M_k(1+t)^{-\frac{1+k}{2}}
    \end{equation*}
    for some constants $M_k$ for any $0 \leq k \leq 2$.
\end{lemma}
\begin{proof}
    The long time case is directly from \cite[Theorem II]{kato1994navier} and the short time case is proved by the parabolic maximum principle, see for instance \cite[Lemma 2.2]{feng2023quantitative}.
\end{proof}

Next we shall apply the classical Carlen-Loss argument \cite{carlen1996optimal} to provide some Gaussian upper bounds and $L^p$ norm bounds with respect to the variable $x$ on the limit density $g(t,m,x)$, especially giving the explicit decay rate of those quantities in time $t$, which are optimal if one checks with the easiest heat kernel case. This result is due to the fact that the drift term $K \ast \omega$ is divergence-free and satisfies the long time asymptotic decay rate as in the above Lemma \ref{omega}.

\begin{lemma}[$L^p$ Bounds and Gaussian Upper Bound]\label{gaussianabove}
    Assume that $g(t,m,x)$ solves the nonlinear Fokker-Planck limit equation \eqref{fp} with initial data $g_0 \in W_x^{2,1} \cap W_x^{2,\infty}(\mathbb{D})$. Assume as in our main results that $g_0$ satisfies the Gaussian upper bound
    \begin{equation*}
        g_0(m,x) \leq C_0 \exp(-C_0^{-1} |x|^2)
    \end{equation*}
    for some $C_0>0$. Then we have
    \begin{equation*}
        g(t,m,x) \leq \frac{C}{1+t} \exp\Big(-\frac{|x|^2}{8t+C}\Big)
    \end{equation*}
    for some $C>0$, and for any $1 \leq p \leq \infty$,
    \begin{equation*}
        \|g(t,m,\cdot)\|_{L^p} \lesssim (1+t)^{-1+\frac{1}{p}}.
    \end{equation*}
\end{lemma}

\begin{proof}
    Note that the drift coefficient $K \ast \omega$ in \eqref{fp} is divergence-free, with
    \begin{equation*}
        \| K \ast \omega \|_{L^\infty} \lesssim t^{-\frac{1}{2}}
    \end{equation*}
    by Lemma \ref{omega}. Hence we can apply the optimal spatial decay result \cite[Theorem 3]{carlen1996optimal} with $\beta=1/2$ for instance, say
    \begin{equation*}
        g(t,m,x) \leq \frac{C}{t} \int_{\R^2} \exp \Big(-\frac{|x-y|^2}{8t}\Big) g_0(m,y) \ud y.
    \end{equation*}
    By the Gaussian upper bound assumption on $g_0$, we have
    \begin{align*}
        g(t,m,x) \leq \,& \frac{C}{t} \int_{\R^2} \exp\Big(-\frac{|x-y|^2}{8t}-\frac{|y|^2}{C_0}\Big) \ud y\\
        \leq \,& \frac{C}{t} \int_{\R^2} \exp \Big(-\frac{|x|^2}{8t+C_0}\Big) \exp\Big(-\frac{8t+C_0}{8tC_0}|y|^2\Big) \ud y\\
        \leq \,&\frac{C}{1+t} \exp\Big(-\frac{|x|^2}{8t+C_0}\Big).
    \end{align*}
    The long time $L^p$ norm control is obtained by the optimal smoothing result \cite[Theorem 1]{carlen1996optimal}, while the short time $L^p$ norm control follows from the interpolation inequality and the fact that
    \begin{equation*}
        \| g(t,\cdot)\|_{L^1} \equiv 1, \quad \| g(t,\cdot) \|_{L^\infty} \leq \|g_0 \|_{L^\infty} \leq C
    \end{equation*}
    thanks to the parabolic maximum principle. 
\end{proof}

Then we apply the estimates of the fundamental solution to diffusion processes of generalized divergence form from Osada \cite{osada1987diffusion} to provide some Gaussian lower bound on the limit density.

\begin{lemma}[Gaussian Lower Bound]\label{gaussianbelow}
    Assume that $g(t,m,x)$ solves the nonlinear Fokker-Planck limit equation \eqref{fp} with initial data $g_0 \in W_x^{2,1} \cap W_x^{2,\infty}(\mathbb{D})$. Assume as in our main results that $g_0$ satisfies the logarithmic gradient estimate
    \begin{equation*}
        |\nabla \log g_0| \lesssim 1+|x|.
    \end{equation*}
    Then we have
    \begin{equation*}
        g(t,m,x) \geq \frac{1}{C(1+t)}\exp\Big(-C\frac{|x|^2}{1+t}\Big)
    \end{equation*}
    for some $C>0$.
\end{lemma}

\begin{proof}
    It is straightforward to check that the limit equation \eqref{fp} belongs to the generalized divergence form in \cite{osada1987diffusion}, since the diffusion is uniformly elliptic and the drift is divergence-free and uniformly bounded. Then according to \cite{osada1987diffusion}, we have the Gaussian lower bound for the fundamental solution, denoted by $p(s,x,t,y)$:
    \begin{equation*}
        p(s,x,t,y) \geq \frac{1}{C(t-s)} \exp \Big(-C\frac{|x-y|^2}{t-s} \Big).
    \end{equation*}
    The logarithmic gradient estimate assumption, together with the mean value inequality, yields that $g_0$ has Gaussian lower bound
    \begin{equation*}
        g_0(m,x) \geq \lambda^{-1} \exp(-\lambda|x|^2)
    \end{equation*}
    with some $\lambda>0$. We conclude that
    \begin{align*}
        g(t,m,x)=&\, \int_{\mathbb{D}} p(0,x,t,y)g_0(m,y) \ud y\\
        \geq &\, \frac{1}{C\lambda t} \int_{\mathbb{D}} \exp \Big(-C\frac{|x-y|^2}{t}-\lambda|y|^2\Big) \ud y\\
        \geq &\, \frac{1}{C\lambda t}\int_{\mathbb{D}} \exp \Big(-\frac{C\lambda}{C+\lambda t}|x|^2 \Big) \exp \Big( -\frac{C+\lambda t}{t}|y|^2 \Big) \ud y\\
        \geq &\, \frac{1}{C(1+t)} \exp \Big( -\frac{C\lambda}{C+\lambda t}|x|^2 \Big).
    \end{align*}
\end{proof}

\begin{rmk}
    It is straightforward to check that we only need to impose the Gaussian lower bound assumption on the initial data to establish Lemma \ref{gaussianbelow}. The logarithmic gradient estimate assumption is stated to keep consistency with the main result.
\end{rmk}

Finally we present some higher order Sobolev bounds on the limit density $g$ by the parabolic maximum principle. Notice that these bounds are, different from above, far from optimal in time $t$, but this will not do harm to our main results. We only need these results to verify that the higher order Sobolev bounds of $g$ are bounded in any finite time interval.

\begin{lemma}[Higher Order Sobolev Bounds]\label{high}
    Assume that $g(t,m,x)$ solves the nonlinear Fokker-Planck limit equation \eqref{fp} with initial data $g_0 \in W_x^{2,1} \cap W_x^{2,\infty}(\mathbb{D})$. Then we have for any $1 \leq p \leq \infty$,
    \begin{equation*}
        \| \nabla g(t,m, \cdot) \|_{L^p} \leq Ce^{Ct},
    \end{equation*}
    \begin{equation*}
        \| \nabla^2 g(t,m, \cdot) \|_{L^p} \leq C(1+t)e^{Ct}
    \end{equation*}
    for any $t \in [0,T]$.
\end{lemma}

\begin{proof}
    Denote by $\mathcal{L}$ the linear operator that $\mathcal{L}=(K \ast \omega) \cdot \nabla$.
    We propagate $\partial_i g$ from \eqref{fp} that
    \begin{equation*}
        \partial_t(\partial_i g) +\mathcal{L}(\partial_i g)=\sigma \Delta (\partial_i g) -(K \ast \partial_i \omega) \cdot \nabla g.
    \end{equation*}
    By the parabolic maximum principle, the evolution of $\Vert \nabla g \Vert_\infty$ is given by 
    \begin{equation*}
        \frac{\ud}{\ud t} \Vert \nabla g \Vert_{L^\infty} \leq C \Vert K \ast \nabla \omega \Vert_{L^\infty} \|\nabla g\|_{L^\infty} \leq C \Vert \nabla g \Vert_{L^\infty},
    \end{equation*}
    and the evolution of $\Vert \nabla g \Vert_1$ is given by
    \begin{equation*}
        \frac{\ud}{\ud t} \Vert \nabla g \Vert_{L^1} \leq C \Vert K \ast \nabla \omega \Vert_{L^\infty} \Vert \nabla g \Vert_{L^1} \leq C \Vert \nabla g \Vert_{L^1}.
    \end{equation*}
    Hence the $L^p$ norm of $\nabla g$ is bounded by $Ce^{Ct}$ by the Gr\"onwall's inequality and the interpolation inequality. We next propagate $\nabla^2 g$ from \eqref{fp} that
    \begin{equation*}
        \partial_t(\nabla^2 g) +\mathcal{L}(\nabla^2 g)= \sigma \Delta (\nabla^2 g) -2(K \ast \nabla \omega) \cdot \nabla^2 g-(K \ast \nabla^2 \omega) \cdot \nabla g.
    \end{equation*}
    By the  parabolic maximum principle, the evolution of $\Vert \nabla^2 g \Vert_{L^\infty}$ is given by 
    \begin{equation*}
        \frac{\ud}{\ud t} \Vert \nabla^2 g \Vert_{L^\infty} \leq C \Vert K \ast \nabla \omega \Vert_{L^\infty} \|\nabla^2 g\|_{L^\infty}+C\|K \ast \nabla^2 \omega \|_{L^\infty} \| \nabla g \|_{L^\infty} \leq C \Vert \nabla^2 g \Vert_{L^\infty}+Ce^{Ct},
    \end{equation*}
    and the evolution of $\Vert \nabla g \Vert_1$ is given by
    \begin{equation*}
         \frac{\ud}{\ud t} \Vert \nabla^2 g \Vert_{L^1} \leq C \Vert K \ast \nabla \omega \Vert_{L^\infty} \|\nabla^2 g\|_{L^1}+C\|K \ast \nabla^2 \omega \|_{L^\infty} \| \nabla g \|_{L^1} \leq C \Vert \nabla^2 g \Vert_{L^1}+Ce^{Ct}.
    \end{equation*}
    Hence the $L^p$ norm of $\nabla^2 g$ is bounded by $C(1+t)e^{Ct}$ by the Gr\"onwall's inequality and the interpolation inequality.
\end{proof}

\section{Sharp Logarithmic Growth Estimates}\label{log}

In this section, we focus on the key logarithmic growth estimates stated before. As observed in many previous works for instance \cite{feng2023quantitative,carrillo2025relative,rosenzweig2024relative}, the major difficulty one will face when working on the whole space instead of the compact domain like the torus, is the spatial growth estimates for the derivatives of the logarithm of the limit density. Roughly speaking, when controlling the error term appearing in the time evolution of the relative entropy, we need to set up some growth estimates for quantities like $\nabla \log \bar\rho$ and $\nabla^2 \log \bar\rho$, where $\bar\rho$ generally represents the limit density. Such estimates are also crucial when considering the stability of the limit density in the sense of relative entropy, which goes beyond the range of particle approximation. 

For the torus case, as in \cite{jabin2018quantitative,guillin2024uniform,shao2024quantitative}, one may assume that the initial density has some positive lower bound $\lambda>0$. By the parabolic maximum principle, the limit density keeps the lower bound for any time, giving the trivial estimate
\begin{equation*}
    |\nabla \log \bar\rho| \leq \frac{1}{\lambda} \| \bar\rho \|_{W^{1,\infty}}, \quad |\nabla^2 \log \bar\rho| \leq \frac{1}{\lambda} \| \bar\rho \|_{W^{2,\infty}}.
\end{equation*}
This assumption even leads to a time-uniform LSI for the limit density using the classical Bakry-\'Emery arguments, which plays a crucial role in \cite{guillin2024uniform} in proving the uniform-in-time propagation of chaos. 

But for our whole space case, one can never find such positive lower bound for the density. Therefore, the authors introduced some ideas in \cite{feng2023quantitative} to set up the logarithmic gradient and Hessian estimates and started the program of extending the quantitative propagation of chaos results to the whole Euclidean space. These results were obtained using the famous Li-Yau estimate \cite{Li1986OnTP} and Hamilton estimate \cite{hamilton1993matrix,han2016upper}, and the Bernstein method with the Grigor'yan maximum principle \cite{grigor2006heat,grigoryan2009heat,li2014liyauhamilton}. A simple parabolic maximum principle method is available for the Landau equation with Maxwellian molecules \cite{carrillo2025relative}. In this section, we shall set up the logarithmic estimates for the equation \eqref{fp}, and we have simplified our method using only the Grigor'yan parabolic maximum principle, thanks to the Gaussian upper and lower bound results in Lemma \ref{gaussianabove} and Lemma \ref{gaussianbelow}. We have also optimized the time dependence of the coefficients in the estimate to be sharp in time, improving the bounds in \cite{feng2023quantitative}.

\begin{prop}[Sharp Logarithmic Growth Estimates]\label{logarithmic}
    Assume that $g(t,m,x)$ solves the nonlinear Fokker-Planck limit equation \eqref{fp} with initial data $g_0 \in W_x^{2,1} \cap W_x^{2,\infty}(\mathbb{D})$. Assume as in our main results that $g_0$ satisfies the logarithmic growth estimates
    \begin{equation*}
        |\nabla \log g_0(m,x)|^2 \lesssim 1+|x|^2,
    \end{equation*}
    \begin{equation*}
        |\nabla^2 \log g_0(m,x)| \lesssim 1+|x|^2,
    \end{equation*}
    and the Gaussian upper bound
    \begin{equation*}
        g_0(m,x) \leq C_0\exp(-C_0^{-1}|x|^2),
    \end{equation*}
    then the following logarithmic growth estimates hold for any time $t$,
    \begin{equation}\label{gradient}
        |\nabla \log g(t,m,x)|^2 \leq \frac{B_1}{1+t}(B_2-\log(1+t)-\log g(t,m,x)),
    \end{equation}
    \begin{equation}\label{hessian}
        |\nabla^2 \log g(t,m,x)| \leq \frac{B_3}{1+t}(B_4-\log(1+t)-\log g(t,m,x)),
    \end{equation}
    for some large constants $B_1,B_2,B_3,B_4$. Combining with the Gaussian lower bound in Lemma \ref{gaussianbelow}, we deduce that
    \begin{equation}\label{gradientnew}
        |\nabla \log g(t,m,x)|^2 \lesssim \frac{1}{1+t}+\frac{|x|^2}{(1+t)^2},
    \end{equation}
    \begin{equation}\label{hessiannew}
        |\nabla^2 \log g(t,m,x)| \lesssim \frac{1}{1+t}+\frac{|x|^2}{(1+t)^2}.
    \end{equation}
    These estimates are now sharp in time, since they coincide with the estimates of the heat kernel and the Oseen vortex.
\end{prop}

The basic idea of the proof of is almost the same as in \cite{feng2023quantitative}, namely applying the Bernstein's method by creating some auxiliary functions and propagating them along the parabolic equation \eqref{fp}. Using the specific structure of the coefficients of the equation, we are able to apply the Grigor'yan parabolic maximum principle \cite{grigor2006heat,grigoryan2009heat} to deduce the non-positivity of those auxiliary functions, leading to our desired estimates. However, in order to obtain the sharp time dependence of the estimate, we have modified the proof using some new techniques motivated by \cite{hamilton1993matrix}.

\subsection{Grigor'yan Parabolic Maximum Principle}

 We borrow the following version of Grigor'yan parabolic maximum principle from \cite[Theorem 11.9]{grigoryan2009heat}. This result is originally designed to obtain the uniqueness of the bounded solution to the Cauchy problem, which helps to prove the stochastic completeness of weighted manifolds. We refer to \cite{grigor2006heat,grigoryan2009heat} for the proof of Theorem \ref{gri} and some related discussions.
 
\begin{thm}[Grigor'yan]\label{gri}
    Let $(M,g,e^f\ud V)$ be a complete weighted manifold and let $F(x,t)$ be a solution of
    \begin{equation}\label{condigri1}
        \partial_t F = \Delta_f F \text{ in } M \times (0,T], \quad F(\cdot,0) = 0.
    \end{equation}
    Here $f$ is some bounded function, and the diffusion operator is given by $\Delta_f \triangleq \Delta+\langle \nabla f, \nabla \rangle$.
    Assume that for some $x_0 \in M$ and for all $r>0$,
    \begin{equation}\label{condigri2}
        \int_0^T \int_{B(x_0,r)} F^2(x,t) e^{f(x)} \ud V \ud t \leq e^{\alpha(r)}
    \end{equation}
    for some $\alpha(r)$ positive increasing function on $(0,\infty)$ such that
    \begin{equation}\label{condigri3}
        \int_0^\infty \frac{r}{\alpha(r)} \ud r=\infty.
    \end{equation}
    Then $F = 0$ on $M \times (0,T]$.
\end{thm}

By examining the proof of Theorem \ref{gri}, we notice, as also shown in Y. Li \cite{li2014liyauhamilton}, that the result has the following variant, which we will mainly refer to later.

\begin{thm}[Grigor'yan, an extended version]\label{karp}
    Let $(M,g,e^f\ud V)$ be a complete weighted manifold, and let $F(x,t)$ be a solution of
    \begin{equation}\label{condi1}
        \partial_t F \leq \Delta_f F \text{ in } M \times (0,T], \quad F(\cdot,0) \leq 0.
    \end{equation}
    Assume that for some $x_0 \in M$ and for all $r>0$,
    \begin{equation}\label{condi2}
        \int_0^T \int_{B(x_0,r)} F_+^2(x,t) e^{f(x)} \ud V \ud t \leq e^{\alpha(r)}
    \end{equation}
    for some $\alpha(r)$ positive increasing function on $(0,\infty)$ such that
    \begin{equation}\label{condi3}
        \int_0^\infty \frac{r}{\alpha(r)} \ud r=\infty.
    \end{equation}
    Then $F \leq 0$ on $M \times (0,T]$.
\end{thm}

For the sake of completeness, we explain here why Theorem \ref{karp} is valid. Indeed, since $\phi(x)=x_+$ is non-decreasing, convex, continuous and piece-wise smooth, the condition (\ref{condi1}) yields that
\begin{equation*}
    \partial_t F_+ \leq \Delta_f F_+ \text{ in } M \times (0,T], \quad F_+(\cdot,0) = 0.
\end{equation*}
Also in Theorem \ref{gri} the result still holds for $F$ satisfying $\partial_t F\leq \Delta_f F$ and $F \geq 0$, since the only step using the PDE itself is the equation \cite[(11.37)]{grigoryan2009heat}, which turns into an inequality in the new case and still yields the following steps. Therefore, we apply Theorem \ref{gri} to $F_+$ and obtain Theorem \ref{karp}. Furthermore, it is straightforward to check from the original proof that replacing $\Delta$ by $\sigma \Delta$ in the definition of $\Delta_f$ will not change the result as long as $\sigma>0$. In our case, it is convenient to choose
\begin{equation*}
    f=-h \ast \omega,
\end{equation*}
where
\begin{equation*}
    h(x)=-\frac{1}{2\pi}\arctan\frac{x_1}{x_2}
\end{equation*}
is bounded and satisfies $K=\nabla h$.

\subsection{Logarithmic Gradient Estimate}

Now we turn to the proof of our logarithmic estimates, i.e.  Proposition \ref{logarithmic}. In this subsection, we focus on the gradient estimate \eqref{gradient}, and the Hessian estimate \eqref{hessian} will be discussed in the next subsection. 

First, we provide a technical lemma concerning the propagation results of some quantities.

\begin{lemma}\label{cal} 
Assume that $g$ solves the nonlinear Fokker-Planck limit equation \eqref{fp} and recall that $\Delta_f=\sigma \Delta+\langle \nabla f, \nabla \rangle$ with $\nabla f=-K \ast \omega$. Then we have
\begin{align}
    &(\partial_t -\Delta_f) \Big(\frac{|\nabla g|^2}{g}\Big)=-\frac{2\sigma}{g}\Big|\partial_{ij} g-\frac{\partial_i g\partial_j g}{g}\Big|^2-\frac{2}{g} \nabla g \cdot \nabla (K\ast \omega) \cdot \nabla g \leq \frac{2M_1}{1+t} \frac{|\nabla g|^2}{g},\label{nablagsquare}\\
    &(\partial_t-\Delta_f) (g\log g)=-\sigma\frac{|\nabla g|^2}{g}.\label{glogg}
\end{align}
We recall that the constant $M_1$ is defined in Lemma \ref{omega}.
\end{lemma}

The proof of Lemma \ref{cal} is nothing but technical computations, which we omit here and refer to \cite[Section 6]{feng2023quantitative} for details.

\textbf{Short Time Estimate:} We construct the first auxiliary function to be propagated:
\begin{equation*} 
    F(t,x)=\frac{|\nabla g|^2}{g}+\frac{C_1}{\sigma}g\log g-C_2 g
\end{equation*}
with some constants $C_1,C_2$ to be determined. From Lemma \ref{cal} above, the propagation of the auxiliary function reads
\begin{equation*}
    (\partial_t-\Delta_f)F \leq \frac{2M_1}{1+t}\frac{|\nabla g|^2}{g}-C_1\frac{|\nabla g|^2}{g} \leq 0
\end{equation*}
as long as $C_1 \geq 2M_1$. The initial value of $F$ simply reads
\begin{equation*}
    F(0,x)=\frac{|\nabla g_0|^2}{g_0}+\frac{C_1}{\sigma}g_0\log g_0-C_2 g_0.
\end{equation*}
Hence, $F(0,\cdot) \leq 0$ is equivalent to
\begin{equation*}
    |\nabla \log g_0|^2+\frac{C_1}{\sigma}\log g_0 \leq C_2.
\end{equation*}
We recall the initial assumptions:
\begin{equation*}
    |\nabla \log g_0(m,x)| \lesssim 1+|x|,\quad g_0(m,x) \leq C_0 \exp(-C_0^{-1}|x|^2),
\end{equation*}
which allow us to choose $C_1$ large enough to eliminate the spatial growth in $|\nabla \log g_0|^2$ and then choose $C_2$ large enough to conclude that $F(0,\cdot) \leq 0$. Now we only have to check the growth assumption \eqref{condi2}. Recall that $\nabla g \in L^\infty$ and $g \in L^\infty$ from Lemma \ref{high}, and that
\begin{equation*}
    g(x,t) \geq \frac{1}{C(1+t)}\exp\Big(-C\frac{|x|^2}{1+t}\Big)
\end{equation*}
from the Gaussian lower bound Lemma \ref{gaussianbelow}. Consequently, we check the growth assumption by
\begin{equation*}
    \int_0^T \int_{B_r} F_+^2(t,x) e^{f(x)} \ud V\ud t \leq Te^{C_T(1+r^2)} \int_{B_r} e^{f(x)} \ud x \leq CTr^2e^{C_T(1+r^2)}
\end{equation*}
since $f$ is bounded. Hence we may choose $\alpha(r)=C_Tr^2(1+|\log r|)$ which satisfies the Osgood type condition (\ref{condi3}). Applying the Grigor'yan maximum principle Theorem \ref{karp} we arrive at $F \leq 0$, or
\begin{equation}\label{F}
    |\nabla \log g|^2+\frac{C_1}{\sigma}\log g \leq C_2.
\end{equation}

\textbf{Long Time Estimate:} In order to obtain the logarithmic gradient estimate \eqref{gradient} with optimal time dependence, we need a small trick motivated by \cite{hamilton1993matrix}. For any fixed time $\tau>0$, we define the following time-rescaling function $\Tilde{g}$ by pushing forward the time variable by $\tau/2$:
\begin{equation*}
    \Tilde{g}(t,m,x)=g(t+\frac{\tau}{2},m,x),
\end{equation*}
which also solves the same nonlinear Fokker-Planck equation \eqref{fp} with the initial data changing into $\Tilde{g}(0, \cdot)=g(\frac{\tau}{2},\cdot)$. Hence the corresponding asymptotic decay estimates read as
\begin{equation*}
    \|\nabla^k \Tilde{g}\|_{L^\infty} \leq M_k(1+\frac{\tau}{2}+t)^{-1-\frac{k}{2}}, \quad   \|\nabla^k K \ast \Tilde{g}\|_{L^\infty} \leq M_k(1+\frac{\tau}{2}+t)^{-\frac{1+k}{2}}.
\end{equation*} 
We also need to construct another auxiliary function to be propagated:
\begin{equation*}
    F(t,x)=\phi \frac{|\nabla \Tilde{g}|^2}{\Tilde{g}}+\Tilde{g} 
    \log \Tilde{g} -C_4\Tilde{g},
\end{equation*}
for some $\phi(t)$ and constant $C_4$ to be determined. We deduce that the propagation of the new auxiliary function reads
\begin{equation*}
    (\partial_t-\Delta_f)F \leq \Big(\phi^\prime +\frac{2M_1}{1+t}\phi-\sigma\Big) \frac{|\nabla \Tilde{g}|^2}{\Tilde{g}}.
\end{equation*}
Hence, if we choose $C_4=\log \frac{2M_0}{\tau}$ and $\phi$ satisfying
\begin{equation*}
    \phi^\prime+\frac{2M_1}{1+t}\phi-\sigma \leq 0, \quad \phi(0)=0,
\end{equation*}
then the initial assumption \eqref{condi1} is satisfied immediately. Obviously $\phi(t)=\frac{\sigma t}{2M_1+1}$ works. The rest of the assumptions of Theorem \ref{karp} follow as in the short time case. Therefore, we arrive at $F \leq 0$, or
\begin{equation*}
    |\nabla \log \Tilde{g}|^2+\frac{2M_1+1}{\sigma t} \log \Tilde{g} \leq \frac{C_4(2M_1+1)}{\sigma t}.
\end{equation*}
Take $t=\tau/2$ and return to the original function $g$:
\begin{equation}\label{FF}
    |\nabla \log g(\tau,m,x)|^2+\frac{4M_1+2}{\sigma \tau} \log g(\tau,m,x)\leq \frac{(4M_1+2)(\log 2M_0-\log \tau)}{\sigma\tau}.
\end{equation}

Combining \eqref{F} for $t \leq 1$ and \eqref{FF} for $t \geq 1$, we can choose some appropriate large constants $B_1,B_2$ to conclude the logarithmic gradient estimate \eqref{gradient}.

\subsection{Logarithmic Hessian Estimate}

Now we turn to the Hessian estimate \eqref{hessian} in this subsection. First we also provide a technical lemma concerning the propagation results of some quantities.

\begin{lemma}\label{cal2}  
    Assume that $g$ solves the nonlinear Fokker-Planck equation \eqref{fp} and recall that $\Delta_f=\sigma \Delta+\langle \nabla f, \nabla \rangle$ with $\nabla f =-K \ast \omega$. Then
    \begin{align}
        \label{hessiang}&(\partial_t -\Delta_f) \Big(\frac{|\nabla^2 g|^2}{g}\Big) \leq \frac{4M_1+M_2}{1+t}\frac{|\nabla^2 g |^2}{g}+\frac{M_2}{(1+t)^2}\frac{|\nabla g|^2}{g}.\\
        \label{gloggsquare}&(\partial_t -\Delta_f) \Big(g(\log g)^2\Big)=-2\sigma\frac{|\nabla g|^2}{g}(1+\log g).
    \end{align}
    We recall that $M_1, M_2$ are defined in Lemma \ref{omega}.
\end{lemma}

The proof of Lemma \ref{cal2} is also nothing but technical computations, which we also omit here and refer to \cite[Section 7]{feng2023quantitative} for details.

\textbf{Short Time Estimate:} We construct the first auxiliary function to be propagated:
\begin{equation*}
    F(t,x)=e^{-(4M_1+M_2)t}\frac{|\nabla^2 g|^2}{g}-C_5g(\log g)^2+C_6 g\log g-C_7 g
\end{equation*}
with some constants $C_5,C_6,C_7$ to be determined. From Lemma \ref{cal2} above, the propagation of the auxiliary function reads
\begin{align*}
    (\partial_t-\Delta_f)F \leq &\, -(4M_1+M_2)e^{-(4M_1+M_2)t}\frac{|\nabla^2 g|^2}{g}\\
    &\, +\frac{4M_1+M_2}{1+t}e^{-(4M_1+M_2)t}\frac{|\nabla^2 g|^2}{g}+\frac{M_2}{(1+t)^2}e^{-(4M_1+M_2)t}\frac{|\nabla g|^2}{g}\\
    &\, +2C_5\sigma\frac{|\nabla g|^2}{g}(1+\log g)-C_6\sigma\frac{|\nabla g|^2}{g}\\
    \leq &\, (M_2+2C_5\sigma+2C_5\sigma \log(1+M_0)-C_6\sigma)\frac{|\nabla g|^2}{g}.
\end{align*}
The initial value of $F$ simply reads
\begin{equation*}
    F(0,x)=\frac{|\nabla^2 g_0|^2}{g_0}-C_5g_0(\log g_0)^2+C_6g_0\log g_0-C_7g_0.
\end{equation*}
Hence, $F(0,\cdot) \leq 0$ is equivalent to
\begin{equation*}
    \frac{|\nabla^2 g_0|^2}{g_0^2}+C_6\log g_0 \leq C_7+C_5(\log g_0)^2.
\end{equation*}
We recall the initial assumptions:
\begin{equation*}
    |\nabla \log g_0(m,x)|^2,|\nabla^2 \log g_0(m,x)| \lesssim 1+|x|^2,\quad g_0(m,x) \leq C_0 \exp(-C_0^{-1}|x|^2),
\end{equation*}
which allow us to choose $C_5$ large enough to eliminate the spatial growth in $|\nabla^2 g_0|^2/g_0^2$, and then choose $C_6$ large enough to satisfy that $(\partial_t-\Delta_f)F \leq 0$, and finally choose $C_7$ large enough to conclude that $F(0,\cdot) \leq 0$. Now we only have to check the growth assumption \eqref{condi2}. Recall that $\nabla^2 g \in L^\infty$, $\nabla g \in L^\infty$ and $g \in L^\infty$ from Lemma \ref{high}, and that
\begin{equation*}
    g(x,t) \geq \frac{1}{C(1+t)}\exp\Big(-C\frac{|x|^2}{1+t}\Big)
\end{equation*}
from the Gaussian lower bound Lemma \ref{gaussianbelow}. Consequently, we check the growth assumption by
\begin{equation*}
    \int_0^T \int_{B_r} F_+^2(x,t) e^{f(x)} \ud V\ud t \leq CTe^{C(1+T)(1+r^2)} \int_{B_r} e^{f(x)} \ud x \leq CTr^2e^{C_Tr^2}
\end{equation*}
since $f$ is bounded.
Hence we may choose $\alpha(r)=C_Tr^2(1+|\log r|)$ as well which satisfies the Osgood type condition \eqref{condi3}. Applying the Grigor'yan maximum principle Theorem \ref{karp} we arrive at $F \leq 0$, or
\begin{equation}\label{F2}
    \frac{|\nabla^2 g|^2}{g^2} \leq e^{(4M_1+M_2)t}\Big( C_5 (\log g)^2-C_6\log g+C_7\Big).
\end{equation}

\textbf{Long Time Estimate:} In order to obtain the logarithmic Hessian estimate \eqref{hessian} with optimal time dependence, we also need the small trick motivated by \cite{hamilton1993matrix}. For any fixed time $\tau>0$, we define the time-rescaling function $\Tilde{g}$ as in the previous subsection. We also need to make use of the non-positive term which is abandoned in \eqref{nablagsquare} when discussing the logarithmic gradient estimate. Rewrite \eqref{nablagsquare} into
\begin{equation}\label{newnablasquare}
    (\partial_t-\Delta_f)\Big(\frac{|\nabla g|^2}{g}\Big) \leq -\sigma\frac{|\nabla^2 g|^2}{g}+2\sigma\frac{|\nabla g|^4}{g^3}+\frac{2M_1}{1+t}\frac{|\nabla g|^2}{g}.
\end{equation}
We construct our new auxiliary function to be propagated:
\begin{equation*}
    F(t,x)=\phi \frac{|\nabla^2 \Tilde{g}|^2}{\Tilde{g}}+\psi \frac{|\nabla \Tilde{g}|^2}{\Tilde{g}}-C_8\Tilde{g}(\log \Tilde{g})^2+C_9\Tilde{g} \log \Tilde{g} -C_{10}\Tilde{g},
\end{equation*}
for some $\phi(t)$ and $\psi(t)$ and constants $C_8,C_9,C_{10}$ to be determined. We deduce that the propagation of the new auxiliary function reads
\begin{align*}
    &\, (\partial_t-\Delta_f)F\\
    \leq &\, \Big(\phi^\prime+\frac{4M_1+M_2}{1+t}\phi-\sigma\psi\Big)\frac{|\nabla^2 \Tilde{g}|^2}{\Tilde{g}}\\
    +&\, \Big(\frac{M_2}{(1+t)^2}\phi+\psi^\prime+\frac{2M_1}{1+t}\psi+2\sigma\psi|\nabla \log \Tilde{g}|^2+2\sigma C_8\log \Tilde{g}+2\sigma C_8-\sigma C_9\Big)\frac{|\nabla \Tilde{g}|^2}{\Tilde{g}}.
\end{align*}
Recall the logarithmic gradient estimate 
\begin{equation*}
    \frac{\sigma t}{2M_1+1} |\nabla \log \Tilde{g}|^2+\log \Tilde{g} \leq C_4=\log 2M_0-\log \tau.
\end{equation*}
Now we pick the coefficient function
\begin{equation*}
    \psi(t)=\frac{t}{4M_1+2}
\end{equation*}
to satisfy that
\begin{equation*}
    2\sigma \psi|\nabla \log \Tilde{g}|^2 \leq \log 2M_0-\log \tau-\log \Tilde{g}.
\end{equation*}
We pick the other coefficient function
\begin{equation*}
    \phi(t)=\frac{\sigma t^2}{(4M_1+M_2+2)(4M_1+2)}
\end{equation*}
to cancel the first term on the right-hand side of the propagation of the auxiliary function $F$. Namely, we have
\begin{equation*}
    \phi^\prime+\frac{4M_1+M_2}{1+t}\phi-\sigma \psi\leq \frac{2\sigma t+(4M_1+M_2)\sigma t}{(4M_1+M_2+2)(4M_1+2)}-\frac{\sigma t}{4M_1+2}=0,
\end{equation*}
and we deduce that
\begin{align*}
    (\partial_t-\Delta_f)F \leq \Big(\frac{1}{2}+\frac{\sigma }{4M_1+2}+\log 2M_0-\log \tau+(2\sigma C_8-1)\log \Tilde{g}-\sigma C_9\Big) \frac{|\nabla \Tilde{g}|^2}{\Tilde{g}}.
\end{align*}
The initial value of $F$ simply reads
\begin{equation*}
    F(0,x)=-C_8 \Tilde{g}_0(\log \Tilde{g}_0)^2+C_9 \Tilde{g}_0 \log \Tilde{g}_0-C_{10} \Tilde{g}_0.
\end{equation*}
This allows us to choose $C_8=1/(2\sigma)$ to cancel the term involving $\log \Tilde{g}$, and then choose
\begin{equation*}
    C_9=\frac{1}{2\sigma}+\frac{1}{4M_1+2}+\frac{\log 2M_0}{\sigma}-\frac{\log \tau}{\sigma}=C_9^\prime-\frac{\log \tau}{\sigma}
\end{equation*}
to deduce that $(\partial_t-\Delta_f)F \leq 0$, and finally choose $C_{10}$ large enough to conclude that $F(0,\cdot) \leq 0$ as long as $\tau \geq 1$. The rest of the assumptions of Theorem \ref{karp} follow as in the short time case. Therefore, we arrive at $F \leq 0$, or
\begin{equation*}
    \frac{\sigma t^2}{(4M_1+M_2+2)(4M_1+2)} \frac{|\nabla^2 \Tilde{g}|^2}{\Tilde{g}^2} \leq \frac{(\log \Tilde{g})^2}{2\sigma}-C_9\log \Tilde{g}+C_{10}.
\end{equation*}
Take $t=\tau/2$ and return to the original function $g$:
\begin{equation}\label{t2}
    \frac{|\nabla^2 g|^2}{g}(\tau,m,x) \leq \frac{4(4M_1+M_2+2)(4M_1+2)}{\sigma \tau^2} \Big( \frac{(\log g)^2}{2\sigma}-(C_9^\prime-\log \tau)\log g+C_{10} \Big).
\end{equation}

Combining \eqref{F2} for $t \leq 1$ and \eqref{t2} for $t \geq 1$, we can choose some appropriate large constants $B_3,B_4$ to conclude the logarithmic Hessian estimate \eqref{hessian}.

\medskip 

\noindent {\bf Acknowledgments.} X. F. and Z. W. are partially  supported by the National Key R\&D Program of China, Project Numbers 2021YFA1002800 and 2024YFA1015500,  NSFC grant No.12171009, Young Elite Scientist Sponsorship Program by China Association for Science and Technology (CAST) No. YESS20200028 and the Fundamental Research Funds for the Central Universities (the start-up fund), Peking University.

\noindent {\bf Declaration.} The authors declare that they have no conflict of interest.

\bibliographystyle{abbrv}
\bibliography{ref}

\end{document}